\documentclass[11pt,letterpaper]{article}

\usepackage{fullpage}

\usepackage{graphicx,subfigure}

\def\showauthornotes{2}

\def\showkeys{0}
\def\showdraftbox{1}
\def\showcolorlinks{1}
\def\usemicrotype{1}
\def\showfixme{1}

\def\arxivmode{0}
\def\fastmode{0}

\ifnum\fastmode=0
\usepackage{etex}
\fi

\ifnum\fastmode=0
\usepackage[l2tabu, orthodox]{nag}
\fi

\usepackage{xspace,enumerate}

\usepackage[dvipsnames]{xcolor}

\ifnum\fastmode=0
\usepackage[T1]{fontenc}
\usepackage[full]{textcomp}
\fi

\usepackage[american]{babel}

\usepackage{mathtools}

\usepackage{amsthm}

\newtheorem{theorem}{Theorem}[section]
\newtheorem*{theorem*}{Theorem}

\newtheorem*{proposition*}{Proposition}
\newtheorem{lemma}[theorem]{Lemma}
\newtheorem*{lemma*}{Lemma}
\newtheorem{corollary}[theorem]{Corollary}
\newtheorem*{conjecture*}{Conjecture}

\newtheorem*{fact*}{Fact}

\newtheorem*{exercise*}{Exercise}

\newtheorem*{hypothesis*}{Hypothesis}

\theoremstyle{definition}
\newtheorem{definition}[theorem]{Definition}

\newtheorem{exercise-easy}[theorem]{Exercise}
\newtheorem{exercise-med}[theorem]{Exercise}
\newtheorem{exercise-hard}[theorem]{Exercise$^\star$}

\newtheorem*{claim*}{Claim}

\newtheorem{remark}[theorem]{Remark}
\newtheorem*{remark*}{Remark}

\newtheorem*{observation*}{Observation}

\ifnum\arxivmode=0
\usepackage{newpxtext} %
\usepackage{textcomp} %
\usepackage[varg,bigdelims]{newpxmath}
\usepackage[scr=rsfso]{mathalfa}%
\usepackage{bm} %
\let\mathbb\varmathbb
\fi

\ifnum\arxivmode=1
\usepackage[varg]{pxfonts} %
\usepackage{textcomp} %
\usepackage[scr=rsfso]{mathalfa}%
\usepackage{bm} %
\fi

\ifnum\showkeys=1
\usepackage[color]{showkeys}
\fi

\makeatletter
\@ifpackageloaded{hyperref}
{}
{\usepackage[pagebackref]{hyperref}}

\definecolor{bleudefrance}{rgb}{0.01, 0.1, 1.0}
\definecolor{azure}{rgb}{0.0, 0.5, 1.0}

\ifnum\showcolorlinks=1
\hypersetup{
pagebackref=true,
colorlinks=true,
urlcolor=blue,
linkcolor=blue,
citecolor=OliveGreen}
\fi

\ifnum\showcolorlinks=0
\hypersetup{
pagebackref=true,
colorlinks=false,
pdfborder={0 0 0}}
\fi

\usepackage{prettyref}

\newcommand{\savehyperref}[2]{\texorpdfstring{\hyperref[#1]{#2}}{#2}}

\newrefformat{eq}{\savehyperref{#1}{\textup{(\ref*{#1})}}}
\newrefformat{lem}{\savehyperref{#1}{Lemma~\ref*{#1}}}
\newrefformat{def}{\savehyperref{#1}{Definition~\ref*{#1}}}
\newrefformat{obs}{\savehyperref{#1}{Observation~\ref*{#1}}}
\newrefformat{ass}{\savehyperref{#1}{Assumption~\ref*{#1}}}
\newrefformat{thm}{\savehyperref{#1}{Theorem~\ref*{#1}}}
\newrefformat{cor}{\savehyperref{#1}{Corollary~\ref*{#1}}}
\newrefformat{cha}{\savehyperref{#1}{Chapter~\ref*{#1}}}
\newrefformat{sec}{\savehyperref{#1}{Section~\ref*{#1}}}
\newrefformat{app}{\savehyperref{#1}{Appendix~\ref*{#1}}}
\newrefformat{tab}{\savehyperref{#1}{Table~\ref*{#1}}}
\newrefformat{fig}{\savehyperref{#1}{Figure~\ref*{#1}}}
\newrefformat{hyp}{\savehyperref{#1}{Hypothesis~\ref*{#1}}}
\newrefformat{alg}{\savehyperref{#1}{Algorithm~\ref*{#1}}}
\newrefformat{rem}{\savehyperref{#1}{Remark~\ref*{#1}}}
\newrefformat{item}{\savehyperref{#1}{Item~\ref*{#1}}}
\newrefformat{step}{\savehyperref{#1}{step~\ref*{#1}}}
\newrefformat{conj}{\savehyperref{#1}{Conjecture~\ref*{#1}}}
\newrefformat{fact}{\savehyperref{#1}{Fact~\ref*{#1}}}
\newrefformat{prop}{\savehyperref{#1}{Proposition~\ref*{#1}}}
\newrefformat{prob}{\savehyperref{#1}{Problem~\ref*{#1}}}
\newrefformat{claim}{\savehyperref{#1}{Claim~\ref*{#1}}}
\newrefformat{relax}{\savehyperref{#1}{Relaxation~\ref*{#1}}}
\newrefformat{red}{\savehyperref{#1}{Reduction~\ref*{#1}}}
\newrefformat{part}{\savehyperref{#1}{Part~\ref*{#1}}}
\newrefformat{ex}{\savehyperref{#1}{Exercise~\ref*{#1}}}
\newrefformat{property}{\savehyperref{#1}{Property~\ref*{#1}}}

\newcommand{\Sref}[1]{\hyperref[#1]{\S\ref*{#1}}}

\usepackage{nicefrac}

\ifnum\fastmode=0
\ifnum\usemicrotype=1
\usepackage{microtype}
\fi
\fi

\ifnum\showauthornotes=2
\newcommand{\mynotes}[1]{{\sffamily\small\color{teal}{#1}}\medskip}
\newcommand{\Authornote}[2]{{\sffamily\small\color{Maroon}{[#1: #2]}}\medskip}
\newcommand{\Authornotecolored}[3]{{\sffamily\small\color{#1}{[#2: #3]}}}
\newcommand{\Authorcomment}[2]{{\sffamily\small\color{gray}{[#1: #2]}}}
\newcommand{\Authorstartcomment}[1]{\sffamily\small\color{gray}[#1: }

\newcommand{\Authorfnote}[2]{\footnote{\color{red}{#1: #2}}}
\newcommand{\Authorfixme}[1]{\Authornote{#1}{\textbf{??}}}
\newcommand{\Authormarginmark}[1]{\marginpar{\textcolor{red}{\fbox{\Large #1:!}}}}
\newcommand{\myexplain}[1]{{\sffamily\small\color{red}{\noindent [Explanation:\medskip\newline \begin{quote}#1\hfill]\end{quote}}}\medskip}
\newcommand{\explain}[1]{{\sffamily\small\color{red}{#1}}\medskip}

\else

\newcommand{\mynotes}[1]{}
\newcommand{\Authornote}[2]{}
\newcommand{\Authornotecolored}[3]{}
\newcommand{\Authorcomment}[2]{}
\newcommand{\Authorstartcomment}[1]{}

\newcommand{\Authorfnote}[2]{}
\newcommand{\Authorfixme}[1]{}
\newcommand{\Authormarginmark}[1]{}
\newcommand{\myexplain}[1]{}
\newcommand{\explain}[1]{}

\ifnum\showauthornotes=1
\renewcommand{\myexplain}[1]{{\sffamily\small\color{red}{\noindent \begin{quote}{\bf Explanation:} \medskip\newline #1\end{quote}}}\medskip}
\fi

\fi

\ifnum\showfixme=0

\fi

\usepackage{boxedminipage}

\newcommand{\Esymb}{\mathbb{E}}
\newcommand{\Psymb}{\mathbb{P}}

\DeclareMathOperator*{\E}{\Esymb}

\DeclareMathOperator*{\ProbOp}{\Psymb}

\renewcommand{\Pr}{\ProbOp}

\newcommand{\textparen}[1]{\text{(#1)}}

\ifx\because\undefined
\newcommand{\because}[1]{\textparen{because #1}}
\else
\renewcommand{\because}[1]{\textparen{because #1}}
\fi

\newcommand{\seteq}{\mathrel{\mathop:}=}

\newcommand{\bigmid}{~\big|~}

\newcommand\bdot\bullet

\ifx\mathds\undefined %

\else

\fi

\DeclareMathOperator{\vol}{vol}

\DeclareMathOperator{\dist}{dist}

\newcommand{\Z}{\mathbb Z}
\newcommand{\N}{\mathbb N}
\newcommand{\R}{\mathbb R}
\newcommand{\C}{\mathbb C}

\newcommand{\cA}{\mathcal A}

\newcommand{\cE}{\mathcal E}

\newcommand{\cG}{\mathcal G}

\newcommand{\cM}{\mathcal M}

\newcommand{\cS}{\mathcal S}

\newcommand{\cV}{\mathcal V}

\newcommand{\cX}{\mathcal X}

\renewcommand{\leq}{\leqslant}

\renewcommand{\geq}{\geqslant}

\ifnum\showdraftbox=1

\else

\fi

\let\epsilon=\varepsilon

\numberwithin{equation}{section}

\newcommand\MYcurrentlabel{xxx}

\newcommand{\MYstore}[2]{%
  \global\expandafter \def \csname MYMEMORY #1 \endcsname{#2}%
}

\newcommand{\MYload}[1]{%
  \csname MYMEMORY #1 \endcsname%
}

\newcommand{\MYnewlabel}[1]{%
  \renewcommand\MYcurrentlabel{#1}%
  \MYoldlabel{#1}%
}

\newcommand{\MYdummylabel}[1]{}

\newcommand{\torestate}[1]{%
  \let\MYoldlabel\label%
  \let\label\MYnewlabel%
  #1%
  \MYstore{\MYcurrentlabel}{#1}%
  \let\label\MYoldlabel%
}

\newcommand{\restatetheorem}[1]{%
  \let\MYoldlabel\label
  \let\label\MYdummylabel
  \begin{theorem*}[Restatement of \prettyref{#1}]
    \MYload{#1}
  \end{theorem*}
  \let\label\MYoldlabel
}

\newcommand{\restatelemma}[1]{%
  \let\MYoldlabel\label
  \let\label\MYdummylabel
  \begin{lemma*}[Restatement of \prettyref{#1}]
    \MYload{#1}
  \end{lemma*}
  \let\label\MYoldlabel
}

\newcommand{\restateprop}[1]{%
  \let\MYoldlabel\label
  \let\label\MYdummylabel
  \begin{proposition*}[Restatement of \prettyref{#1}]
    \MYload{#1}
  \end{proposition*}
  \let\label\MYoldlabel
}

\newcommand{\restatefact}[1]{%
  \let\MYoldlabel\label
  \let\label\MYdummylabel
  \begin{fact*}[Restatement of \prettyref{#1}]
    \MYload{#1}
  \end{fact*}
  \let\label\MYoldlabel
}

\newcommand{\restate}[1]{%
  \let\MYoldlabel\label
  \let\label\MYdummylabel
  \MYload{#1}
  \let\label\MYoldlabel
}

\newcommand{\addreferencesection}{
  \phantomsection
\ifnum\stocmode=0
  \addcontentsline{toc}{section}{References}
\else
  \addcontentsline{toc}{section}{References \hspace*{1in} --------- End of extended abstract ---------}
\fi

}

\newcommand{\e}{\epsilon}

\renewcommand{\paragraph}[1]{\medskip\noindent{\bf #1.}}

\allowdisplaybreaks

\usepackage{paralist}

\usepackage{comment}

\let\pref=\prettyref

\newcommand{\diam}{\mathrm{diam}}

\newcommand{\vertiii}[1]{{\left\vert\kern-0.25ex\left\vert\kern-0.25ex\left\vert #1 
          \right\vert\kern-0.25ex\right\vert\kern-0.25ex\right\vert}}

\usepackage{enumitem}
\makeatletter
\def\namedlabel#1#2{\begingroup
    #2%
    \def\@currentlabel{#2}%
    \phantomsection\label{#1}\endgroup
}
\makeatother

\usepackage[normalem]{ulem}
\usepackage{stmaryrd}
\usepackage{accents}
\newcommand{\ubar}[1]{\underaccent{\bar}{#1}}

\newcommand{\lra}{\leftrightarrow}
\newcommand{\energy}{\mathscr{E}}
\newcommand{\ase}{\textrm{{\em a.s.e.}}}
\newcommand{\spase}{\ \ase}

\newcommand{\cmnt}[1]{}

\newcommand{\reff}{\mathsf{R}_{\mathrm{eff}}}

\newcommand{\green}{\mathsf{g}}
\newcommand{\Green}{\mathsf{Gr}}

\newcommand{\1}{\mathbb{1}}

\renewcommand{\mathbb}{\vvmathbb}

\newcommand{\rgraphs}{\mathscr{G}_{\bullet}}
\newcommand{\rrgraphs}{\mathscr{G}_{\bullet\bullet}}
\newcommand{\Mod}{\mathsf{Mod}}

\newcommand{\rball}{B^G}
\newcommand{\crball}{\bar{B}^G}
\newcommand{\con}{c}
\newcommand{\amod}{\mathsf{M}}

\newcommand{\up}[1]{\bar{#1}}
\newcommand{\down}[1]{\ubar{#1}}

\newcommand*{\diffeo}{%
  \mathrel{\vcenter{\offinterlineskip
  \hbox{$\sim$}\vskip-.35ex\hbox{$\sim$}\vskip-.35ex\hbox{$\sim$}}}}

\usepackage{import}
\usepackage{xifthen}
\usepackage{pdfpages}
\usepackage{transparent}

\begin{document}

\title{Relations between scaling exponents \\ in unimodular random graphs}

\author{James R. Lee\thanks{University of Washington}}
\date{}

\maketitle

\begin{abstract}
   We investigate the validity of the ``Einstein relations''
   in the general setting of unimodular random networks.
   These are equalities relating scaling exponents:
   \begin{align*}
      d_w &= d_f + \tilde{\zeta}, \\
      d_s &= 2 d_f/d_w,
   \end{align*}
   where $d_w$ is the walk dimension, $d_f$ is the fractal dimension, 
   $d_s$ is the spectral dimension,
   and $\tilde{\zeta}$ is the resistance exponent.
   Roughly speaking, this relates the mean displacement and return probability
   of a random walker to the density and conductivity of the underlying medium.
   We show that if $d_f$ and $\tilde{\zeta} \geq 0$ exist, then
   $d_w$ and $d_s$ exist, and the aforementioned equalities hold.
   Moreover, our primary new estimate $d_w \geq d_f + \tilde{\zeta}$,
   is established for all $\tilde{\zeta} \in \R$.

   For the uniform infinite planar triangulation (UIPT),
   this yields the consequence
   $d_w=4$ using $d_f=4$ (Angel 2003) and $\tilde{\zeta}=0$ (established here
   as a consequence of the Liouville Quantum Gravity theory, following Gwynne-Miller 2020
   and Ding-Gwynne 2020).
   The conclusion $d_w=4$ had been previously established by Gwynne and Hutchcroft (2018) using
   more elaborate methods.
   A new consequence is that $d_w = d_f$ for the uniform infinite Schnyder-wood decorated triangulation,
   implying that the simple random walk is subdiffusive, since $d_f > 2$.

   For the random walk on $\Z^2$ driven by conductances from an
   exponentiated Gaussian free field with
   exponent $\gamma > 0$, one has $d_f = d_f(\gamma)$ and $\tilde{\zeta}=0$ (Biskup, Ding, and Goswami 2020).
   Thus the scaling relations yield $d_s=2$ and $d_w = d_f$, confirming two predictions of those authors.
\end{abstract}

\newpage

\begingroup
\hypersetup{linktocpage=false}
\setcounter{tocdepth}{2}
\tableofcontents
\endgroup

\section{Introduction}

Consider an infinite, locally-finite graph $\cG$ and a subgraph $G$ of $\cG$.
For $x \in V(\cG)$, let $B^{\cG}(x,R)$, denote the graph ball of radius $R$,
and let $\tilde{B}(x,R) \seteq B^{\cG}(x,R) \cap V(G)$ denote this ball restricted to $G$.
Let $d^{\cG}(x,y)$ denote
the path distance between a pair $x,y \in V(\cG)$.
Denote by $\{X_n\}$ the simple random walk on $G$, and the discrete-time heat kernel
\[
   p^G_n(x,y) \seteq \Pr[X_n = y \mid X_0 = x].
\]
We write $\reff^G(S \leftrightarrow T)$ for the effective resistance between
two subsets $S,T \subseteq V(G)$.

For a variety of models arising in statistical physics, certain asymptotic
geometric and spectral properties of the graph are known or conjectured
to have scaling exponents:
\begin{align}
   |\tilde{B}(x,R)| &\sim R^{d_f} \nonumber \\
   \max_{1 \leq t \leq n} d^{\cG}(X_0,X_t) &\sim n^{1/d_w} \nonumber\\
   \reff^G\left(\tilde{B}(x,R) \leftrightarrow V(G) \setminus \tilde{B}(x, 2R)\right) &\sim R^{\tilde{\zeta}} \label{eq:resist-sim}  \\
   p_{2n}^G(x,x) &\sim n^{-d_s/2},\nonumber
\end{align}
where one takes $n,R \to \infty$, but
we leave the meaning of ``$\sim$'' imprecise for a moment.
These exponents are, respectively, referred to as the {\em fractal dimension}, {\em walk dimension},
{\em spectral dimension}, and {\em resistance exponent}.
We refer to the extensive discussion in \cite[Ch. 5--6]{BH00}.

Moreover, by modeling the subgraph $G$ as a homogenous underlying substrate
with density and conductivity prescribed by $d_f$ and $\tilde{\zeta}$, one obtains
the plausible relations
\begin{align}
   d_w &= d_f + \tilde{\zeta} \label{eq:e1}\\
   d_s &= \frac{2 d_f}{d_w}. \label{eq:e2}
\end{align}

In the regime $\tilde{\zeta} > 0$, these relations have been rigorously verified under somewhat stronger assumptions
in the setting of {\em strongly recurrent graphs} (see \cite{Telcs90,Telcs95}) and \cite{Barlow98,KM08,Kumagai14}).
In the latter set of works, the most significant departure from our assumptions is the stronger
requirement for uniform control on pointwise effective resistances of the form
\begin{equation}\label{eq:reff-strong}
   \max \left\{ \reff^G(x \lra y) : y \in B^G(x,R) \right\} \leq R^{\tilde{\zeta}+o(1)}, \quad x\in V(G).
\end{equation}
Such methods have been extended to the setting where $(G,\rho)$ is a random rooted graph (\cite{KM08,BJKS08})
under the statistical assumption that these relations hold sufficiently often for all sufficiently large scales,
and only for balls around the root.
Our main contribution
is to establish \eqref{eq:e1} and \eqref{eq:e2} under somewhat less restrictive conditions, but
using an additional feature of many such models:  Unimodularity of the random rooted graph $(G,\rho)$.
When $\tilde{\zeta} \leq 0$,
it has been significantly more challenging to characterize
situations where \eqref{eq:e1}--\eqref{eq:e2} hold; see, for instance, Open Problem III in \cite{KumagaiICM}.
Our main new estimate is the speed relation $d_w \geq d_f + \tilde{\zeta}$,
which is established for all $\tilde{\zeta} \in \R$.
In particular, this shows that the random walk is subdiffusive whenever $d_f+\tilde{\zeta} > 2$,
and applies equally well to models where the random walk is transient.
Let us now highlight some notable settings in which the relations can be applied.

\paragraph{The IIC in high dimensions}
As a prominent example,
consider the resolution by Kozma and Nachmias \cite{kn09}
of the Alexander-Orbach conjecture for the incipient infinite cluster (IIC) of critical
percolation on $\Z^d$, for $d$ sufficiently large.
If $(G,0)$ denotes the IIC, then in our language, $\cG=G$, as they consider
the intrinsic graph metric, and 
establish that for every $\lambda > 1$ and $r \geq 1$,
with probability at least $1-p(\lambda)$, it holds that
\begin{gather}
   \lambda^{-1} r^2 \leq |B^{G}(0,r)| \leq \lambda r^2, \\
   \reff^{G}(0 \leftrightarrow \partial B^{G}(0,r)) \geq \lambda^{-1} r,
\end{gather}
where $p(\lambda) \leq O(\lambda^{-q})$ for some $q > 1$.
One should consider this a statistical verification that $d_f = 2$ and $\tilde{\zeta} = 1$,
as in this setting, one gets the analog of \eqref{eq:reff-strong} for free
free from the trivial bound $\reff^{\textrm{IIC}}(0 \lra x) \leq d^{\textrm{IIC}}(0,x)$.

Earlier, Barlow, J\'arai, Kumagai, and Slade \cite{BJKS08} verified \eqref{eq:e1}--\eqref{eq:e2} under
these assumptions, allowing Kozma and Nachmias to confirm the conjectured values $d_w = 3$ and $d_s = 4/3$.
One can consult \cite[\S 4.2.2]{KumagaiICM} for several further examples where $\tilde{\zeta} > 0$
and \eqref{eq:e1}--\eqref{eq:e2} hold using the strongly recurrent theory.

\paragraph{The uniform infinite planar triangulation}
Consider, on the other hand, the uniform infinite planar triangulation (UIPT) considered
as a random rooted graph $(G,\rho)$.  
In this case, Angel \cite{angel03} established that almost surely
\begin{equation}\label{eq:angel}
   \lim_{R \to \infty} \frac{\log |B^G(\rho,R)|}{\log R} = 4,
\end{equation}
and Gwynne and Miller \cite{GM21} showed that almost surely
\begin{equation*}\label{eq:gm-reff}
   \lim_{R \to \infty} \frac{\log \reff^G(\rho \leftrightarrow V(G) \setminus B^G(\rho,R))}{\log R} = 0\,.
\end{equation*}
This falls short of verifying \eqref{eq:resist-sim}.
Nevertheless, we show in \pref{sec:LQG} that $\tilde{\zeta}=0$ is a consequence of the
Liouville Quantum Gravity (LQG) estimates derived in \cite{DMS14,GM21,GMS19,GHS20,DG20}.
But while the known statistics of $|B^G(\rho,R)|$ are suitable to allow application of the
strongly recurrent theory, this does not hold for the effective resistance bounds.

This is highlighted by Gwynne and Hutchcroft \cite{GH20} who establish $d_w = 4$
using even finer aspects of the LQG theory.
The authors state
``while it may be possible in principle to prove $\beta \geq 4$ using electrical techniques,
doing so appears to require matching upper and lower bounds for effective resistances [...]
differing by at most a constant order multiplicative factor.''
Our methods show that, when leveraging unimodularity, even coarse estimates
with subpolynomial errors suffice.

It is open whether $\tilde{\zeta} = 0$ or $d_w = 4$ for the uniform infinite planar {\em quadrangulation} (UIPQ),
but our verification of \eqref{eq:e1} shows that only one such equality needs to be established.

\paragraph{Random planar maps in the $\gamma$-LQG universality class}
More generally, we will establish in \pref{sec:LQG} that
$\tilde{\zeta}=0$ whenever a random planar map
$(G,\rho)$ can be coupled to a $\gamma$-mated-CRT map with $\gamma \in (0,2)$.
The connection between such maps and LQG was established in \cite{DMS14}.

This family includes the UIPT (where $\gamma = \sqrt{8/3}$).  Ding and Gwynne \cite{DG20} have shown that
$d_f$ exists for such maps, and Gwynne and Huthcroft \cite{GH20} established that $d_w=d_f$ for most known examples,
but not for the uniform infinite Schnyder-wood decorated triangulation \cite{LSW17} (where $\gamma=1$),
for a technical reason underlying the construction of a certain coupling (see \cite[Rem. 2.11]{GH20}).
We mention this primarily to emphasize the utility of a general theorem, since
it is likely the technical obstacle could have been circumvented with sufficient effort.

\paragraph{Random walk driven by a Gaussian free field}
Biskup, Ding, and Goswami \cite{BDG20} study the model of random walk on $\cG = \Z^2$
with random conductances $\con^G(\{u,v\}) = e^{\gamma (\eta_v - \eta_u)}$, where
$\gamma > 0$, and $\{\eta_v : v \in \Z^2\}$  is the discrete Gaussian free field (GFF) 
on $\Z^2$ grounded at the origin.

In this case, one has
\begin{equation*}\label{eq:gff-df}
   d_f = \begin{cases}
            2 + 2(\gamma/\gamma_c)^2 & \gamma \leq \gamma_c = \sqrt{\pi/2}, \\
            4 \gamma/\gamma_c & \textrm{otherwise.}
         \end{cases}
\end{equation*}
(See below for the definition of $d_f$ when the edges have conductances.)

In \pref{sec:gff}, we recall the model formally and
observe that the paper \cite{BDG20} contains estimates that
establish $\tilde{\zeta} = 0$ for every $\gamma > 0$.
Hence the relations \eqref{eq:e1}--\eqref{eq:e2}
yield $d_f=d_w$ and $d_s=2$, both of which were conjectured
in \cite{BDG20}, though only annealed estimates were obtained.
(See \pref{sec:quenched} for a brief discussion of why 
our approach yields two-sided quenched bounds.)

\paragraph{The IIC in dimension two}
Consider the incipient infinite cluster for 2D critical percolation \cite{Kesten86iic}, which can
be realized as a unimodular random subgraph $(G,0)$ of $\cG=\Z^2$ \cite{Jarai03}.
It is known that $d_f = 91/48$ in the 2D hexagonal lattice \cite{LSW02,Smirnov01}, and the same
value is conjectured to hold for all 2D lattices regardless of the local structure.

Existence of the exponent $\tilde{\zeta}$ is open for any lattice; experiments
give the estimate $\tilde{\zeta} = 0.9825 \pm 0.0008$ \cite{Grassberger99}.
The most precise experimental estimate for $d_w = 2.8784 \pm 0.0008$ is derived
from estimates for $\tilde{\zeta}$, and our verfication of \eqref{eq:e1} puts
this on rigorous footing (assuming, of course, that $\tilde{\zeta}$ is well-defined).
Indeed, one motivation for our work was the question of whether the exponent
$d_w$ should be a conformal invariant of critical 2D percolation,
and it is plausibly more tractable to establish this for $\tilde{\zeta}$.

\subsection{Reversible random networks}
\label{sec:stationary}

We consider random rooted networks $(G,\rho,\con^G,\xi)$ where
$G$ is a locally-finite, connected graph, $\rho \in V(G)$, and $\con^G : E(G) \to [0,\infty)$ are edge conductances.
We allow $E(G)$ to contain self-loops $\{v,v\}$ for $v \in V(G)$.
Here, $\xi : V(G) \cup E(G) \to \Xi$ is an auxiliary marking, where $\Xi$ is some Polish mark space.
We will sometimes use the notation $(G,\rho,\xi_1,\xi_2,\ldots,\xi_k)$ to reference a random rooted network
with marks $\xi_i : V(G) \cup E(G) \to \Xi_i$, which we intend as shorthand for $(G,\rho,(\xi_1,\xi_2,\ldots,\xi_k))$,
where the mark space is the Cartesian product $\Xi_1 \times \cdots \times \Xi_k$.

Denote by $\{X_n\}$ the random walk on $G$ with $X_0=\rho$ and transition probabilities
\begin{equation}\label{eq:transition}
   p^G_n(u,v) \seteq \Pr\left[X_{n+1} = v \mid X_n = u\right] = \frac{\con^G(\{u,v\})}{\con^G_u},
\end{equation}
where we denote $\con^G_u \seteq \sum_{v : \{u,v\} \in E(G)} \con^G(\{u,v\})$.
Say that $(G,\rho,\con^G,\xi)$ is a {\em reversible random network} if:
\begin{enumerate}
   \item Almost surely $\con^G_{\rho} > 0$.
   \item $(G,X_0,X_1,\con^G,\xi)$ and $(G,X_1,X_0,\con^G,\xi)$ have the same law.
\end{enumerate}
We will usually write a reversible random network as $(G,\rho,\xi)$, allowing
the conductances to remain implicit.
Note that we allow the possibility $\con^G(\{u,v\})=0$ when $\{u,v\} \in E(G)$.
In this sense, random walks occur on the subnetwork $G_+$ with $V(G_+) = \{ x \in V(G) : \con^G_x > 0 \}$
and $E(G_+) = \{ \{x,y\} \in V(G) : \con^G(\{x,y\}) > 0\}$, while distances are measured in 
the path metric $d^G$.

\smallskip 

Throughout, we will make the following mild boundedness assumption:
\[
   \E[1/\con^G_{\rho}] < \infty\,.
\]
This is analogous to the assumption $\E[\deg_G(\rho)] < \infty$ that appears
often in the setting of unimodular random graphs, which are defined
in \pref{sec:mtp} when we need to employ the Mass-Transport Principle.

For now, it suffices to say that there is a one-to-one correspondence:
\[
\begin{matrix}
   (G,\rho,\xi) \textrm{ reversible} & \longleftrightarrow & (\tilde{G},\tilde{\rho},\tilde{\xi}) \textrm{ unimodular} \\
   \E[1/\con^G_{\rho}] < \infty & & \E[\con^{\tilde{G}}_{\tilde{\rho}}] < \infty\vphantom{\frac{\bigoplus}{\bigoplus}}
\end{matrix}
\]
Indeed, if $\mu$ and $\tilde{\mu}$ are the respective measures, then the correspondence
is given by a change of law
\[
   \frac{d\mu}{d\tilde{\mu}}(G_0,\rho_0,\xi_0) = \frac{\con^{G_0}_{\rho_0}}{\E[\con^{\tilde{G}}_{\tilde{\rho}}]}\,.
\]
where $d\mu/d\tilde{\mu}$ is the Radon-Nikodym derivative.
We refer to \cite{aldous-lyons} for an extensive reference on unimodular random graphs,
and to \cite[Prop. 2.5]{bc12} for the connection between unimodular and reversible random graphs.

\subsection{Almost sure scaling exponents}

Consider two sequences $\{A_n\}$ and $\{B_n\}$ of positive real-valued random variables.
Write $A_n \lessapprox B_n$ if almost surely:
\[
   \limsup_{n \to \infty} \frac{\log A_n - \log B_n}{\log n} \leq 0,
\]
and $A_n \diffeo B_n$ for the conjuction of $A_n \lessapprox B_n$ and $B_n \lessapprox A_n$.
Note that $A_n \lessapprox n^d$ if and only if, for every $\delta > 0$,
almost surely $A_n \leq n^{d+\delta}$ for $n$ sufficiently large.

In what follows, we consider a reversible random network $(G,\rho)$ (cf. \pref{sec:stationary}).
Define the random variables:
\begin{align*}
   \sigma_R &\seteq \min \{ n \geq 0 : d^{G}(X_0,X_n) > R \}, \\
   \cM_n &\seteq \max_{0 \leq t \leq n} d^{G}(X_0,X_t),
\end{align*}
and define the walk exponents $d_w$ and $\beta$ by
\begin{align*}
   \sigma_R  &\diffeo R^{d_w} \\
   \cM_n &\diffeo n^{1/\beta},
\end{align*}
assuming the corresponding limits exist.  In that case we, we will use the
language ``$d_w$ exists'' or ``$\beta$ exists.''\footnote{In the next section,
   we control the annealed variants as well, where one takes
expectations over the random walk.}

Denote the volume function
\[
   \vol^G(x,R) \seteq \sum_{y \in B^G(x,R)} c^G_y,
\]
and define $d_f$ as the asymptotic growth rate of the volume:
\[
   \vol^G(\rho,R) \diffeo R^{d_f},
\]
Define the spectral dimension by
\[
   p_{2n}^G(\rho,\rho) \diffeo n^{-d_s/2}.
\]

Let us define upper and lower resistance exponents.
Define $\tilde{\zeta}$ and $\tilde{\zeta}_0$ as
the largest and smallest values, respectively, such that, for every $\delta > 0$, almost surely,
for all but finitely many $R \in \N$:
\begin{equation}\label{eq:zeta-def}
   R^{\tilde{\zeta}-\delta}  \leq
   \reff^{G}\left(\rball(\rho,R^{1-\delta}) \leftrightarrow \crball(\rho, R)\right) 
   \leq \reff^G\left(\rho \leftrightarrow \crball(\rho,R)\right) \leq R^{\tilde{\zeta}_0+\delta},
\end{equation}
where we have denoted the complement of $\rball(\rho,R)$ in $G$ by
\[
   \crball(\rho,R) \seteq V(G) \setminus \rball(\rho,R).
\]
The exponents $\tilde{\zeta} \leq \tilde{\zeta}_0$ always exist and $\tilde{\zeta}_0 \geq 0$.
The exponent $\tilde{\zeta}$ is referred to as the ``resistance exponent''
in the statistical physics literature (see \cite[\S 5.3]{BH00}); see \pref{rem:resist} below.
We emphasize that all the exponents we define are not random variables, but functions of the law of $(G,\rho)$.
Our main theorem can then be stated as follows.

\begin{theorem}\label{thm:main}
   Suppose that $(G,\rho)$ is a reversible random network satisfying
   $\E[1/c^G_{\rho}] < \infty$.
   If $d_f$ exists and $\tilde{\zeta}=\tilde{\zeta}_0$, then
   the exponents
   $d_w$, $\beta$, and $d_s$ exist and it holds that
   \begin{align*}
      d_w &= \beta = d_f + \tilde{\zeta}, \\
      d_s &= \frac{2 d_f}{d_w}.
   \end{align*}
\end{theorem}

See \pref{cor:main} for further equalities involving annealed versions of $d_w$ and $\beta$.

\begin{remark}[The resistance exponents]
   \label{rem:resist}
   The resistance exponent is usually characterized heuristically as the value $\tilde{\zeta}$
   such
   \begin{equation}\label{eq:other-def}
      \reff^G\left(\rball(\rho,R) \leftrightarrow \crball(\rho,2R)\right) \diffeo R^{\tilde{\zeta}}.
   \end{equation}
   So the left-hand side of \eqref{eq:zeta-def} would naturally be replaced by
   \[
      \reff^G\left(\rball(\rho,R) \leftrightarrow \crball(\rho,2R)\right) \geq R^{\tilde{\zeta}-\delta}.
   \]
   The lower bound we require is substantially weaker, allowing one to consider spatial fluctuations of magnitude $R^{o(1)}$.
   The upper bound in \eqref{eq:zeta-def}, on the other hand, is somewhat stronger
   than \eqref{eq:other-def}, and encodes a level of spectral regularity.
   For instance, if $G$ satisfies an elliptic Harnack inequality and is ``strongly recurrent''
   in the sense of \cite[Def. 2.1]{Telcs06}, then
   \[
      \reff^G(\rball(\rho,R) \leftrightarrow \crball(\rho,2R)) \diffeo 
      \reff^G(\rho \leftrightarrow \crball(\rho,R)).
   \]
   See \cite[Thm. 4.6]{Telcs06} and \pref{thm:sr-compare}.
\end{remark}

\paragraph{Comparison to the strongly recurrent theory}
\label{sec:sr}
   Let us try to interpret the strongly recurrent theory (cf. Assumption 1.2 in \cite{KM08}) 
   in the setting of subpolynomial errors.  The resistance assumptions would take the form:
   For every $\delta > 0$, almost surely, for $R$ sufficiently large:
   \begin{gather}
      \max \left\{ \reff^G(\rho \leftrightarrow x) : x \in \rball(\rho,R) \right\} \leq R^{\zeta+\delta}, \label{eq:sr-ub} \\
      \reff^G\left(\rho \leftrightarrow \crball(\rho,R)\right) \geq R^{\zeta-\delta}. \label{eq:sr-lb}
   \end{gather}
   These assumptions imply that when $\zeta > 0$,
   it holds that $\tilde{\zeta}=\tilde{\zeta}_0=\zeta$; this is proved in \pref{thm:sr-compare}.
   Hence the theory we present (in the setting of unimodular random graphs)
   is more general, at least in terms of concluding the relations \eqref{eq:e1} and \eqref{eq:e2}.

   Under assumptions \eqref{eq:sr-ub} and \eqref{eq:sr-lb}, one can uniformly lower bound the Green kernel
   $\green_{\rball(\rho,R')}(\rho,x)$ (see \pref{sec:resistance} for definitions)
   for all points $x \in \rball(\rho,R)$ and some $R' \gg R$.
   In other words, every point in $\rball(\rho,R)$ is visited often on average before
   the random walk exits $\rball(\rho,R')$.  See, for instance, \cite[\S 3.2]{bck05}.
   This yields a subdiffusive estimate on the speed of the random walk,
   specifically an almost sure lower bound on $\E[\sigma_R \mid (G,\rho)]$.

   Instead of a pointwise bound, we use a lower bound on $\tilde{\zeta}$ to
   deform the graph metric $d^{G}$ (see the next section).
   The effective resistance across an annulus being large is equivalent
   to its discrete extremal length being large (see \pref{sec:modulus}).
   Thus in most scales and localities, we can extract a metric that locally ``stretches''
   the space.  By randomly covering the space with annuli at all scales,
   we obtain a ``quasisymmetric'' deformation (only in an asymptotic, statistical sense) that
   is bigger by a power than the graph metric.
   This argument is similar in spirit to one of Keith and Laakso \cite[Thm. 5.0.10]{KL04}
   which shows that the Assouad dimension of a metric measure space
   can be reduced through a quasisymmetric homeomorphism if the discrete modulus
   across annuli is large.
   
   Finally, by applying Markov type theory,
   we bound the speed of the walk in the stretched metric, which leads to
   a stronger bound in the graph metric.

\subsection{Upper and lower exponents}
\label{sec:sketch}

Even when scaling exponents do not exist, our arguments
give inequalities between various superior and inferior limits.
Given a sequence $\{\cE_n : n \geq 1\}$ of events on some probability space, let us
say that they occur {\em almost surely eventually (a.s.e.)} if
$\Pr[\# \{ n \geq 1 : \neg \cE_n\} < \infty] = 1$.

For a family $\{A_n\}$ of random variables,
we will define $\down{d}$ and $\up{d}$ to be the largest and smallest values,
respectively, such that for every $\delta > 0$,
almost surely eventually,
\[
   n^{\down{d}+\delta} \leq A_n \leq n^{\up{d}+\delta},
\]
where we allow the exponents to take values $\{-\infty,+\infty\}$ if no such number exists.
Note that $A_n \diffeo n^{d}$ (i.e., the exponent $d$ ``exists'') if and only if $\up{d}=\down{d}$.

Let us consider the corresponding extremal exponents
such that
for every $\delta > 0$ the following relations hold almost surely eventually
(with respect to $n,R \geq 1$):
\begin{align*}
   R^{\down{d}_f-\delta} &\leq \vol^G(\rho,R) \leq R^{\up{d}_f+\delta} \\
R^{\down{d}_w-\delta} &\leq \sigma_R \leq R^{\up{d}_w+\delta} \\
R^{\down{d}^{\cA}_w-\delta} &\leq \E[\sigma_R \mid (G,\rho)]  \leq R^{\up{d}^{\cA}_w+\delta} \\
n^{-\delta+1/\up{\beta}} &\leq \cM_n \leq n^{\delta+1/\down{\beta}} \\
n^{-\delta+2/\up{\beta}^{\cA}} &\leq \E[\cM_n^2 \mid (G,\rho)]\leq n^{\delta+2/\down{\beta}^{\cA}} \\
   n^{-\delta-\up{d}_s/2} &\leq p_{2n}^G(\rho,\rho) \leq n^{\delta-\down{d}_s/2},
\end{align*}
We will establish the following chains of inequalities, which together
prove \pref{thm:main}.

\begin{theorem}\label{thm:inequalities}
   Suppose that $(G,\rho)$ is a reversible random network satisfying $\E[1/c^G_{\rho}] < \infty$.
   Then it holds that
\begin{align}\label{eq:relation-mt}
   2\down{d}_f - \up{d}_f + \tilde{\zeta}\ &\leq\ \down{\beta}^{\cA} \\
                                          &\leq\ \down{\beta} \label{eq:bc1} \\
                                          &\leq\ \down{d}_w \wedge \up{\beta} \label{eq:relation-easy} \\
                                          &\leq\ \down{d}_w \vee \up{\beta} \nonumber \\
                                          &\leq \up{d}_w \label{eq:relation-easy2} \\
                                          &\leq\ \up{d}_w^{\cA}  \label{eq:bc2} \\
                                          &\leq\ \up{d}_f + \tilde{\zeta}_0, \label{eq:relation-commute}
\end{align}
and
\begin{equation}\label{eq:relation-ds}
   2 \left(1-\frac{\tilde{\zeta}_0}{\down{d}_w}\right) \leq \down{d}_s \leq\up{d}_s
                                                    \leq \frac{2 \up{d}_f}{\down{d}_w}.
\end{equation}
\end{theorem}
To see that this yields \pref{thm:main},
simply note that when $\tilde{\zeta}=\tilde{\zeta}_0$ and $\down{d}_f=\up{d}_f$, then the upper and lower bounds in
\eqref{eq:relation-commute} and \eqref{eq:relation-mt} match,
and the upper and lower bounds in \eqref{eq:relation-ds}
are both equal to $2 d_f/d_w$ because the first set of inequalities implies $d_w=d_f+\tilde{\zeta}$.

\begin{remark}[Negative resistance exponent]
For $\tilde{\zeta} < 0$, \pref{thm:inequalities} yields (assuming $d_s, d_w$ exist):
\begin{gather*}
   d_w \geq d_f + \tilde{\zeta} \\
   2 \leq d_s \leq \frac{2 d_f}{d_f+\tilde{\zeta}}.
\end{gather*}
Without further assumptions, the final inequality cannot be made tight.
Indeed, for every $\e > 0$,
there are
unimodular random planar graphs of almost sure uniform polynomial growth and $\tilde{\zeta} \leq -1+\e$ \cite{EL20}.
Yet these graphs must satisfy $d_s \leq 2$ \cite{lee17a}.

In the general setting of Dirichlet forms on metric measure spaces,
the ``resistance conjecture'' \cite[pg. 1493]{GHL15} asserts conditions under which \eqref{eq:e1}--\eqref{eq:e2}
might hold even for $\tilde{\zeta} < 0$.
The primary additional condition is a Poincar\'e inequality with matching exponent.
In our setting, the existence of $d_f$ does not yield
the ``bounded covering'' property, that almost surely every ball $B^G(\rho,R)$
can be covered by $O(1)$ balls of radius $R/2$.
It seems likely that a variant of this condition should also be imposed
to recover \eqref{eq:e1}--\eqref{eq:e2}.
\end{remark}

Let us give a brief outline of how \pref{thm:inequalities} is proved.
The unlabeled inequality is trivial.
Both inequalities \eqref{eq:bc1} and \eqref{eq:bc2} are a straightforward consequence of Markov's
inequality and the Borel-Cantelli lemma.
The content of inequalities \eqref{eq:relation-easy} and \eqref{eq:relation-easy2}
lies in the relations $\down{\beta} \leq \down{d}_w$ and $\up{\beta} \leq \up{d}_w$.
These follow from the elementary inequality
\begin{equation}\label{eq:elem1}
   \cM_n \geq \1_{\{\sigma_R\leq n\}} R,
\end{equation}
which gives the implications
\begin{align*}
   \cM_n \leq n^{1/(\down{\beta}-\delta)}\spase &\implies \sigma_R \geq R^{\down{\beta}-\delta}\spase, \\
   \sigma_R \leq R^{\up{d}_w+\delta}\spase &\implies \cM_n \geq n^{1/(\up{d}_w+\delta)}\spase
\end{align*}

Since these hold for every $\delta > 0$, we conclude that
$\up{\beta} \leq \up{d}_w$ and
$\down{d}_w \geq \up{\beta}$, as desired.
Inequalities \eqref{eq:relation-commute} and \eqref{eq:relation-ds}
are proved in \pref{sec:resistance} using the standard relationships between effective resistance, the Green kernel, and return probabilities.
That leaves \eqref{eq:relation-mt}, which relies on Markov type theory, as we now explain.

\paragraph{Reversible random weights}
Consider a reversible random graph $(G,\rho)$ and random edge weights $\omega : E(G) \to \R_+$.
Denote by $\dist_{\omega}^G$ the $\omega$-weighted path metric in $G$.\footnote{Strictly speaking,
since we allow $\omega$ to take the value $0$, this is only a pseudometric,
but that will not present any difficulty.}
When $(G,\rho,\omega)$ is a reversible random network and $(G,\rho)$
is clear from context, we will say simply that the weight $\omega$ is {\em reversible.}

Let us remark that, when $(G,\rho)$ is a reversible random network, then
$(G,\rho,\omega)$ is also a reversible random network if $\omega$
is independent of $\rho$ conditioned on the isomorphism class of $G$.
Our constructions will have this property.
The next theorem (proved in \pref{sec:mt}) is a variant of the approach pursued in \cite{lee17a}.

\begin{theorem}\label{thm:mt}
   Suppose $(G,\rho,\omega)$ is a reversible random network with $\E[1/c^G_{\rho}] < \infty$ and such that
   almost surely
   \begin{equation}\label{eq:subexp}
      \lim_{R \to \infty} \frac{\log \vol^G(\rho,R)}{\log R} = 0.
   \end{equation}
   Suppose, moreover, that
   \begin{equation}\label{eq:omega-L2}
      \E\left[\omega(X_0,X_1)^2\right] < \infty,
   \end{equation}
   where $\{X_n\}$ is random walk on $G$ started from $X_0=\rho$.
   Then it holds that
   \begin{equation}\label{eq:as-diffusive}
      \E\left[\max_{0 \leq t \leq n} \dist_{\omega}^G(X_0,X_t)^2 \mid (G,\rho,\omega)\right] \lessapprox n.
   \end{equation}
\end{theorem}

\begin{figure}[h]
      \centering
      \subfigure[Stretching an annulus\label{fig:stretch}]{\includegraphics[width=6cm]{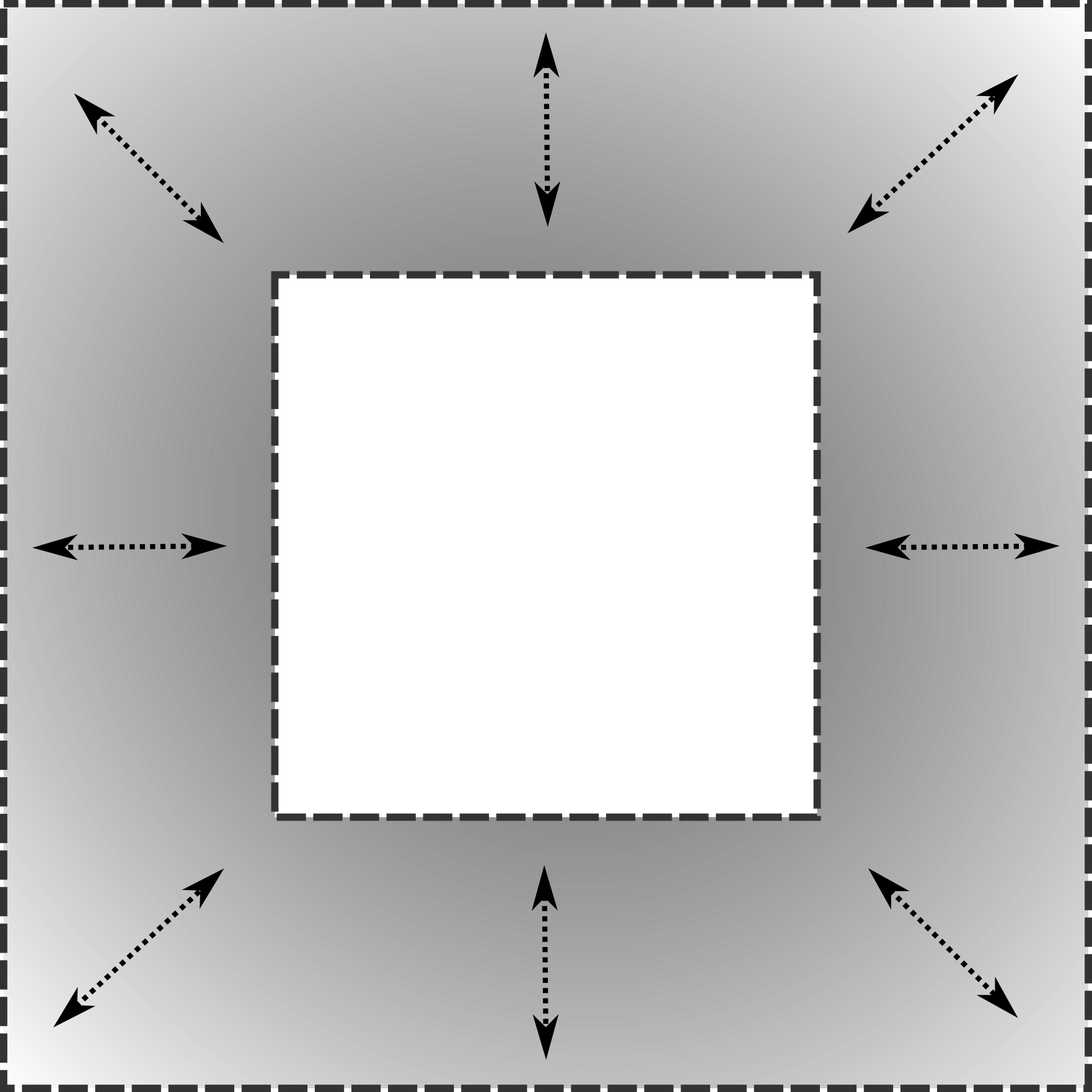}}
         \hspace{0.8in}
         \subfigure[Tiling by annuli\label{fig:tile}]{
   \def\svgwidth{6cm}
%
\begingroup%
  \makeatletter%
  \providecommand\color[2][]{%
    \errmessage{(Inkscape) Color is used for the text in Inkscape, but the package 'color.sty' is not loaded}%
    \renewcommand\color[2][]{}%
  }%
  \providecommand\transparent[1]{%
    \errmessage{(Inkscape) Transparency is used (non-zero) for the text in Inkscape, but the package 'transparent.sty' is not loaded}%
    \renewcommand\transparent[1]{}%
  }%
  \providecommand\rotatebox[2]{#2}%
  \newcommand*\fsize{\dimexpr\f@size pt\relax}%
  \newcommand*\lineheight[1]{\fontsize{\fsize}{#1\fsize}\selectfont}%
  \ifx\svgwidth\undefined%
    \setlength{\unitlength}{316.70078747bp}%
    \ifx\svgscale\undefined%
      \relax%
    \else%
      \setlength{\unitlength}{\unitlength * \real{\svgscale}}%
    \fi%
  \else%
    \setlength{\unitlength}{\svgwidth}%
  \fi%
  \global\let\svgwidth\undefined%
  \global\let\svgscale\undefined%
  \makeatother%
  \begin{picture}(1,1)%
    \lineheight{1}%
    \setlength\tabcolsep{0pt}%
    \put(0,0){\includegraphics[width=\unitlength,page=1]{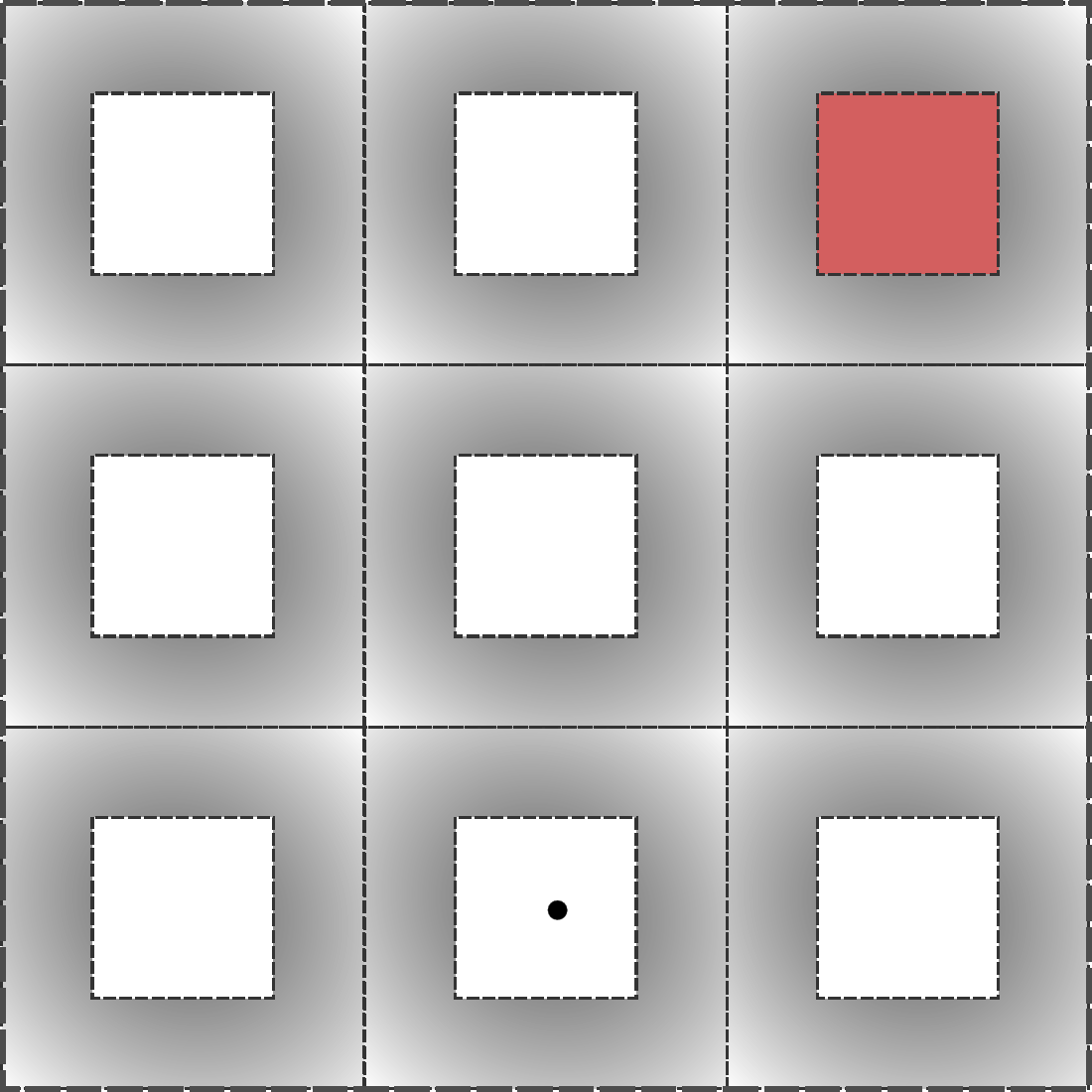}}%
    \put(0.83003842,0.83576267){\color[rgb]{0,0,0}\makebox(0,0)[lt]{\lineheight{1.25}\smash{\begin{tabular}[t]{l}$x$\end{tabular}}}}%
    \put(0.4571141,0.11781708){\color[rgb]{0,0,0}\makebox(0,0)[lt]{\lineheight{1.25}\smash{\begin{tabular}[t]{l}$y$\end{tabular}}}}%
    \put(0,0){\includegraphics[width=\unitlength,page=2]{drawing.pdf}}%
  \end{picture}%
\endgroup%

}
         \caption{Streching the graph at a fixed scale}
\end{figure}

Given this theorem, let us now sketch the proof of \eqref{eq:relation-mt}.
Consider a graph annulus
\[\cA \seteq \{ x \in V(G) : R \leq  d^{G}(\rho,x) \leq R^{1+\delta}\}.\]
If the effective resistance across $\cA$ is at least $R^{\tilde{\zeta}}$, then by the
duality between effective resistance and discrete extremal length (see \pref{sec:modulus}),
there is a length functional $L : E(G[\cA]) \to \R_+$ satisfying
\begin{gather*}
   \sum_{\{x,y\} \in E(G[\cA])} \con^G(\{x,y\}) L(x,y)^2 \leq R^{-\tilde{\zeta}} \\
   \dist_{L}^{G[\cA]}\left(\rball(x,R), \crball(x,R^{1+\delta})\right) \geq 1,
\end{gather*}
where $G[\cA]$ is the subgraph induced on $\cA$.

Let us suppose that the total volume in $\cA$ satisfies
\[
   V_{\cA} \seteq \sum_{e \in E(G[\cA])} \con^G(e) \approx R^{d_f},
\]
and we normalize $L$ to have expectation squared $\leq 1$ under the measure $\con^G(\{x,y\})/V_{\cA}$ on $E(G[\cA])$:
\[
   \hat{L} \seteq R^{(\tilde{\zeta}+d_f)/2} L \approx \left(\frac{R^{\tilde{\zeta}}}{V_{\cA}}\right)^{1/2} L.
\]
This yields:
\[
   \dist_{\hat{L}}^{G[\cA]}\left(\rball(x,R), \crball(x,R^{1+\delta})\right) \geq R^{(\tilde{\zeta}+d_f)/2},
\]
meaning that, with normalized unit area, $\hat{L}$ ``stretches'' the graph annulus by a positive power when $\tilde{\zeta}+d_f > 2$
(see \pref{fig:stretch}).

If $G$ is sufficiently regular (e.g., a lattice), then we could tile annuli at this scale (as in \pref{fig:tile})
so that if we define $\omega_R$ as the sum of the length functionals over the tiled annuli,
then for any pair $x,y \in V(G)$ with $d^{G}(x,y) \geq R^{1+\delta}$
and at least one of $x$ or $y$ in the center of an annulus,
we would have $\dist_{\omega_R}^G(x,y) \geq R^{(\tilde{\zeta}+d_f)/2}$.
In a finite-dimensional lattice, a bounded number of shifts of the tiling
is sufficient for every vertex to reside in the center of some annulus.

\smallskip

By combining length functionals over all scales, and replacing the regular tiling by a suitable random family of annuli,
we obtain, for every $\delta > 0$, a reversible random weight $\omega : E(G) \to \R_+$ satisfying \eqref{eq:omega-L2} (intuitively,
because of the unit area normalization), and such that almost surely eventually
\begin{equation}\label{eq:distpower}
   \dist_{\omega}^G\left(\rho, \crball(\rho,R)\right) \geq R^{(d-\delta)/2},
\end{equation}
where $d \seteq d_f + \tilde{\zeta}$.
In other words, distances in $\dist_{\omega}^G$ are (asymptotically)
increased by power $(d-\delta)/2$.

Thus \eqref{eq:as-diffusive} gives for every $\delta > 0$,
eventually almost surely
\[
   \E\left[\cM_n^2 \mid (G,\rho)\right] \leq n^{2(1+\delta)/(d-\delta)}.
\]
Taking $\delta \to 0$ yields $\down{\beta}^{\cA} \geq d$.
This is carried out formally in \pref{sec:beta}.

\subsubsection{Annealed vs. quenched subdiffusivity}
   \label{sec:quenched}

   One can express $\E[\sigma_R \mid (G,\rho)]$
   in terms of electrical potentials.
   Doing so, it is natural to arive at two-sided annealed estimates:
   \[
      R^{d-o(1)} \leq \E[\sigma_R] \leq R^{d+o(1)} \quad \textrm{as } R \to \infty,
   \]
   where expectation is taken over both the walk and the random network $(G,\rho)$.
   Then a standard application of Borel-Cantelli gives that
   almost surely $\sigma_R \leq R^{d+o(1)}$, but not an almost sure lower bound.
   On the other hand,
   \[
      \E[\cM_n^2] \leq n^{2/d+o(1)}\quad \textrm{as } n \to \infty
   \]
   provides that $\cM_n \leq n^{1/d+o(1)}$ almost surely, which entails $\sigma_R \geq R^{d-o(1)}$ almost surely.

   In this way, the two exponents $\beta$ and $d_w$ are complementary, allowing one
   to obtain two-sided quenched estimates from two-sided annealed estimates.
   This is crucial for establishing $d_s = 2 d_f/d_w$, as the upper bound in
   \eqref{eq:relation-ds} uses the fully quenched exponent $\down{d}_w$
   which, in the setting of \pref{thm:main}, arises from the lower bound \eqref{eq:relation-mt}
   on the (partially) {\em annealed} exponent $\down{\beta}^{\cA}$.
   We remark on the following strengthening of \pref{thm:main}.

   \begin{corollary}\label{cor:main}
      Under the assumptions of \pref{thm:main}, it additionally holds that
      $\beta = \beta^{\cA}$ and $d_w = d_w^{\cA}$.
   \end{corollary}

   \begin{proof}
   We may assume that $d_w$ and $\beta$ exist, and $d_w=\beta$.
   From \pref{thm:inequalities} we obtain:
   \[
      \down{\beta}^{\cA} = \beta = \up{d}_w^{\cA}.
   \]
   The relations $\down{\beta}\leq \down{d}_w^{\cA}$ and $\up{\beta}^{\cA} \leq \up{d}_w$
   follow from \eqref{eq:elem1}, yielding
   \begin{align*}
      \beta &\geq \up{\beta}^{\cA} \geq \down{\beta}^{\cA} = \beta, \\
      \beta &\leq \down{d}_w^{\cA} \leq \up{d}_w^{\cA} = \beta.\qedhere
   \end{align*}
   \end{proof}

\section{Reversible random weights}

Throughout this section, $(G,\rho)$ is a reversible random network satisfying $\E[1/\con^G_{\rho}] < \infty$.

\subsection{Modulus and effective resistance}
\label{sec:modulus}

For a network $H$ and two disjoint subsets $S, T \subseteq V(H)$, define the {\em modulus}
\begin{equation}\label{eq:Hmodulus}
   \Mod^H(S\leftrightarrow T) \seteq \min \left\{ \|\omega\|_{\ell^2(\con^H)}^2 : \dist^H_{\omega}(S,T) \geq 1 \right\},
\end{equation}
where the minimum is over all weights $\omega : E(H) \to \R_+$, and
\[
   \|\omega\|^2_{\ell^2(\con^H)} = \sum_{e \in E(H)} \con^H(e) |\omega(e)|^2.
\]
For $x \in V(H)$ and $0 < r < R$, define the annular modulus:
\[
   \amod^{H}(x,r,R) \seteq \Mod^{H}\left(B^H(x,r) \leftrightarrow \bar{B}^H(x,R)\right).
\]

Note that when $H$ is finite, the minimizer in \eqref{eq:Hmodulus} exists and is unique (as it is the
minimum of a strictly convex function over a compact set).
In particular, even when $H$ is infinite, this also holds for $M^H(x,r,R)$, as we have
\[
\Mod^H\left(B^H(x,r) \lra \bar{B}^H(x,R)\right) = \Mod^{H[B^H(x,R+1)]}\left(B^H(x,r) \lra \bar{B}^H(x,R)\right).
\]
Denote this minimal weight by $\omega_{(H,x,r,R)}$.
The standard duality between effective resistance and discrete extremal length \cite{duffin62}
gives an alternate characterization of $\amod^H(x,r,R)$, as follows.

\begin{lemma}\label{lem:reff-duality}
   For any finite graph $H$ and disjoint subsets $S,T \subseteq V(H)$, it holds that
   \begin{equation}
      \Mod^H(S \leftrightarrow T) 
      = \left(\reff^H(S \leftrightarrow T)\right)^{-1}.
   \end{equation}
   Hence for any (possibly infinite graph) $G$,
   all $x \in V(G)$ and $0 \leq r \leq R$,
   \[
      \amod^H(x,r,R) = \left(\reff^{G}\left(\rball(x,r) \leftrightarrow \crball(x,R)\right)\right)^{-1}.
   \]
\end{lemma}

For a function $g : V(H) \to \R$, we denote the Dirichlet energy
\[
   \energy^H(g) \seteq \sum_{\{x,y\} \in E(H)} \con^H(\{x,y\}) |g(x)-g(y)|^2.
\]
We will make use of the Dirichlet principle (see \cite[Ch. 2]{lp:book}):
When $H$ is finite and $S \cap T = \emptyset$,
   \begin{equation}\label{eq:dirichlet}
      \reff^H(S \leftrightarrow T) = \left(\min \left\{ \energy^H(g) : g|_S \equiv 0, g|_T \equiv 1 \right\}\right)^{-1},
   \end{equation}
and when $H$ is additionally connected, the minimizer of \eqref{eq:dirichlet} is the unique function harmonic on $V(H) \setminus (S \cup T)$
with the given boundary values.

\subsection{Approximate nets}

Fix $R' \geq R \geq 1$ and $\lambda \geq 1$.
For an edge $e \in E(G)$, define
\[
   \gamma_{R,R'}(e) \seteq \max \left\{ \vol^G(y,R) : d^{G}(e,y) \leq 2R' \right\}.
\]
Let $\{\bm{u}_e \in \{0,1\} : e \in E(G)\}$ be an independent family of Bernoulli random variables where
\[
   \Pr\left(\bm{u}_e = 1\right) = \min\left(1,\lambda \frac{\con^G(e)}{\gamma_{R,R'}(e)}\right).
\]
Define $\bm{U}_{R,R'}(\lambda) \seteq \left\{ x \in V(G) : x \in e \textrm{ for some } \bm{u}_e = 1 \right\}$.
Observe the inequalities, valid for every $x \in V(G)$ and $1 \leq r \leq R$:
\begin{align}
   \label{eq:samp1} \Pr[x \in \bm{U}_{R,R'}(\lambda)] &\leq 
   \sum_{y : \{x,y\} \in E(G)} \con^G(\{x,y\}) \gamma_{R,R'}(\{x,y\}) \nonumber \\ &\leq
    \frac{\lambda\con^G_x}{\max \{ \vol^G(y, R) : y \in \rball(x,R') \}}, \\
   \label{eq:samp2} \Pr[d^{G}(x, \bm{U}_{R,R'}(\lambda)) > r] &\leq \prod_{e \in E^G(\rball(x,r))} \left(1-\frac{\lambda \con^G(e)}{\gamma_{R,R'}(e)}\right)_+ 
\leq \exp\left(-\lambda \frac{\vol^G(x,r)}{\vol^G(x,3R')}\right),\vphantom{\frac{\bigoplus}{\bigoplus}}
   \end{align}
   where we use $E^G(B^G(x,R))$ to denote the set of edges in $G$ incident
   to at least one vertex of $B^G(x,R)$.

   \smallskip

   The idea here is that, by \eqref{eq:samp2}, the balls $\{ \rball(u,R) : u \in \bm{U}_{R,R'}(\lambda) \}$ tend to cover
   vertices $x \in V(G)$ for which
   $\vol^G(x,R) \approx \vol^G(x,3R')$, as long as $\lambda$ is chosen sufficiently large.
   On the other hand, \eqref{eq:samp1} will allow us to bound
   $\E |\rball(\rho,2R') \cap \bm{U}_{R,R'}(\lambda)|$.
   Referring to the argument sketched at the end of \pref{sec:sketch},
   we will center an annulus at every $x \in \bm{U}_{R,R'}(\lambda)$, and thus we need
   to control the average covering multiplicity to keep $\E[\omega(X_0,X_1)^2]$ finite.

   Since the law of $\bm{U}_R(\lambda)$ does not depend on the root, we have the following.

\begin{lemma}
   The triple $(G,\rho,\bm{U}_R(\lambda))$ is a reversible random network.
\end{lemma}

\subsection{The Mass-Transport Principle}
\label{sec:mtp}

Let $\rgraphs$ denote the collection of isomorphism classes
of rooted, connected, locally-finite networks, and let $\rrgraphs$
denote the collection of isomorphism classes of doubly-rooted, connected, locally-finite networks.
We will consider functionals $F : \rrgraphs \to [0,\infty)$.
Equivalently, these are functionals $F(G_0,x_0,y_0,\xi_0)$ that are invariant
under automorphisms of $\psi$ of $G_0$:
$F(G_0,x_0,y_0,\xi_0)=F(\psi(G_0),\psi(x_0),\psi(y_0),\xi_0 \circ \psi^{-1})$.

The {\em mass-transport principle (MTP)} for a random rooted network $(G,\rho,\xi)$ asserts
that for any nonnegative Borel $F : \rrgraphs \to [0,\infty)$, it holds that
\begin{equation*}
   \E\left[\sum_{x \in V(G)} F(G,\rho,x,\xi)\right]
   =
   \E\left[\sum_{x \in V(G)} F(G,x,\rho,\xi)\right].
\end{equation*}
Unimodular random networks are precisely those that satisfy the MTP (see \cite{aldous-lyons}).

Using the fact that biasing the law of a reversible random network $(G,\rho,\xi)$ with $\E[1/\con^G_{\rho}] < \infty$
by $1/\con^G_{\rho}$ (see \cite[Prop. 2.5]{bc12}) yields a unimodular random network,
one arrives at the following biased MTP.

\begin{lemma}\label{lem:mtp}
   If $(G,\rho,\xi)$ is a reversible random network with $\E[1/\con^G_{\rho}] < \infty$, then
   for any nonnegative Borel functional $F : \rrgraphs \to [0,\infty)$, it holds that
\begin{equation}\label{eq:masstp}
   \E\left[\frac{1}{\con^G_{\rho}} \sum_{x \in V(G)} F(G,\rho,x,\xi)\right]
   =
   \E\left[\frac{1}{\con^G_{\rho}} \sum_{x \in V(G)} F(G,x,\rho,\xi)\right].
\end{equation}
\end{lemma}

\subsection{Construction of the weights}

Recall that $(G,\rho)$ is a reversible random network satisfying $\E[1/\con^G_{\rho}] < \infty$.
Denote $d_{*} \seteq 2 \down{d}_f - \up{d}_f+\tilde{\zeta}$.
Our goal is to prove the following.

\begin{theorem}\label{thm:good-weight}
   For every $\delta > 0$, there is a reversible random weight $\omega : E(G) \to \R_+$
   such that
   $\E[\omega(X_0,X_1)^2] < \infty$,
   and almost surely eventually
   \begin{equation}\label{eq:good-weight}
      \dist_{\omega}^G\left(\rho, \crball(\rho,R)\right) \geq R^{(d_*-\delta)/2}.
   \end{equation}
\end{theorem}
This this end,
for $\e \in (0,1)$, define the set of networks with controlled geometry at scale $R$:
\begin{align*}
   \cS(\e,R) \seteq \left\{ (G,x) : 
      \vphantom{\frac{\con^G(\rball(x,R-1)}{\con^G(\rball(x,18R)} \geq R^{-\e}}
   \right. & \frac{1+\vol^G(x,5 R^{1+\e})}{\vol^G(x,R)^2} \amod^G(x, 2R, R^{1+\e}) \leq R^{-d_{*}+2\e}
 \\
           & \textrm{and } \left.\frac{\vol^G(x,R-1)}{\vol^G(x,15 R^{1+\e})} \geq d_* R^{-2\e} \log R \right\}
\end{align*}

\begin{lemma}\label{lem:good-weights}
   For every $\e > 0$ ,
   there is a reversible random weight $\omega_R : E(G) \to \R_+$ such that
\begin{equation}\label{eq:mass-ub}
   \E\left[\omega_R(X_0,X_1)^2\right] \leq 2 R^{-d_{*}+4\e},
\end{equation}
and if $x \in V(G)$ satisfies $d^{G}(\rho,x) \geq 3 R^{1+\e}$, then
\begin{equation}\label{eq:dist-lb}
   \dist_{\omega_R}^G(\rho,x) \geq \1_{\cS(\e,R)}(G,\rho).
\end{equation}
\end{lemma}

Before proving the lemma, let us see that it establishes \pref{thm:good-weight}.

\begin{proof}[Proof of \pref{thm:good-weight}]
   Clearly we may assume $d_* > 0$.
   Fix a value $\e \in (0,d_{*})$,
   and define the sets
\begin{align*}
   \cS_{R_0}(\e) &\seteq \bigcap_{R \geq R_0} \cS(\e,R)\,, \\
      \cS(\e) &\seteq \bigcup_{R_0 \geq 1} \cS_{R_0}(\e)\,.
\end{align*}

\begin{lemma}\label{lem:eventually}
   Almost surely $(G,\rho) \in \cS(\e)$.
\end{lemma}

\begin{proof}
   By definition of the exponent $\zeta$, for every $\delta > 0$, it holds that almost surely eventually
   $\amod^G(\rho,R,R^{1+\delta}) \leq R^{\zeta+\delta}$ (recall \pref{lem:reff-duality})
   and $R^{\down{d}_f-\delta} \leq \vol^G(\rho,R) \leq R^{\up{d}_f+\delta}$.
\end{proof}

For $k \geq 1$, let $\omega_{2^k}$ be the weight guaranteed by \pref{lem:good-weights}, and
define the random weight
\[
   \omega \seteq \left(\sum_{k \geq 1} \frac{2^{k(d_{*}-4\e)}}{k^2} \omega_{2^k}^2\right)^{1/2},
\]
so that
\[
   \E\left[\omega(X_0,X_1)^2\right] \stackrel{\eqref{eq:mass-ub}}{\leq} 2 \sum_{k \geq 1} k^{-2} \leq O(1).
\]
Moreover, for any $k \geq 1$ and $x \in V(G)$, if $d^{G}(\rho,x) \geq 3 \cdot 2^{k(1+\e)}$, then \eqref{eq:dist-lb} gives
\[
   \dist_{\omega}^G(\rho,x) \geq k^{-1} 2^{k(d_{*}-4 \e)/2} \dist_{\omega_{2^k}}^G(\rho,x) \geq
   k^{-1} 2^{k(d_{*}-4\e)/2} \1_{\cS_{2^k}(\e)}(G,\rho),
\]
hence for all $x \in V(G)$,
\[
   \dist_{\omega}^G(\rho,x) \geq \frac{\left(d^{G}(\rho,x)/3-3\cdot 2^{k(1+\e)}\right)_+^{(d_{*}-4\e)/(2(1+\e))}}{2 \log d^{G}(\rho,x)} \1_{\cS_{2^k}(\e)}(G,\rho).
\]
Now by \pref{lem:eventually}, this shows that almost surely eventually (with respect to $R \geq 1$):
\[
   d^{G}(\rho,x) \geq R \implies \dist_{\omega}^G(\rho,x) \geq d^{G}(\rho,x)^{(d_{*}-5\e)/(2(1+\e))}.
\]
Since we can take $\e > 0$ arbitrarily small, the desired result follows.
\end{proof}

Let us now prove the lemma.

\begin{proof}[Proof of \pref{lem:good-weights}]
   Fix $R \geq 1$, and define
\[ 
   \cS'(\e,R) \seteq \left\{ z \in V(G) : \frac{1+\vol^G(z,4 R^{1+\e})}{\left(\max \{ \vol^G(y,R) : y \in \rball(z,R) \}\right)^2} \amod^G(z,R,2 R^{1+\e}) \leq R^{-d_{*}+2\e} \right\}.
\]

\begin{lemma}\label{lem:close}
   If $(G,\rho) \in \cS(\e,R)$ and $d^{G}(\rho,z) \leq R$, then $z \in \cS'(\e,R)$.
\end{lemma}

\begin{proof}
   Note that $d^{G}(\rho,z) \leq R$ gives
   \[
      \amod^G(z,R,2 R^{1+\e}) \leq \amod^G(\rho,2R,R^{1+\e}).
   \]
   Similarly, we have $\vol^G(z,4 R^{1+\e}) \leq \vol^G(\rho,5 R^{1+\e})$, and
   \[
      \max \left\{ \vol^G(y,R) : y \in \rball(z,R) \right\} \geq \vol^G(\rho,R).\qedhere
   \]
\end{proof}

Denote $R' \seteq 5 R^{1+\e}$ and, recalling \pref{sec:modulus},
\[
   \omega^{(z)} \seteq \omega_{(G,z,R,2 R^{1+\e})} \1_{\cS'(\e,R)}(z)\,.
\]
Then define:
$\omega_R : E(G) \to \R_+$ by
\begin{align*}
   \hat{\omega} &\seteq \sum_{z \in \bm{U}_{R,R'}(\lambda)} \omega^{(z)}\,, \\
   \tilde{\omega}(\{x,y\}) & \seteq \begin{cases}
      1 & \{x,y\} \nsubseteq \rball(\bm{U}_{R,R'}(\lambda),R) \textrm{ and } \{(G,x),(G,y)\} \cap \cS(\e,R) \neq \emptyset \\
      0 & \textrm{otherwise.}
   \end{cases} \\
   \omega_R &\seteq \hat{\omega} + \tilde{\omega}.
\end{align*}

\begin{lemma}
   If $x \in V(G)$ satisfies $d^{G}(\rho,x) \geq 3 R^{1+\e}$, then $\dist^G_{\omega_R}(\rho,x) \geq \1_{\cS(\e,R)}(G,\rho)$.
\end{lemma}

\begin{proof}
   If $d^{G}(\rho,\bm{U}_{R,R'}(\lambda)) > R$ and $(G,\rho) \in \cS(\e,R)$, then
   $\tilde{\omega}(\{\rho,y\}) \geq 1$ for every $\{\rho,y\} \in E(G)$, implying
   $\dist^G_{\tilde{\omega}}(\rho,x) \geq 1$.

   So suppose that $z \in \bm{U}_{R,R'}(\lambda)$ satisfies $d^{G}(\rho,z) \leq R$.
   By \pref{lem:close}, we have $z \in \cS'(\e,R)$, and therefore $\hat{\omega} \geq \omega_{(G,z,R,2 R^{1+\e})}$.
   Thus by definition,
   \[
      \dist_{\omega_R}^G(\rho,x) \geq \dist^G_{\omega_{(G,z,R,2 R^{1+\e})})}\left(\rball(z,R), \crball(z,2 R^{1+\e})\right) \geq 1,
   \]
   since $\rho \in \rball(z,R)$, and $x \notin \rball(z,2R^{1+\e})$.
\end{proof}

So we are left only to evaluate $\E[\omega_R(X_0,X_1)^2]$.
Use Cauchy-Schwarz to bound
\begin{align}
   \E\left[\hat{\omega}(X_0,X_1)^2\right] &=\E\left[\left(\sum_{z \in \bm{U}_{R,R'}(\lambda)} \omega^{(z)}(X_0,X_1)\right)^2\right]  \nonumber \\
                                          &\leq \E\left[\left|\rball(\rho,2 R^{1+\e}) \cap \bm{U}_{R,R'}(\lambda)\right| \sum_{z \in \rball(X_0,2 R^{1+\e})} \1_{\bm{U}_{R,R'}(\lambda)}(z)\, \omega^{(z)}(X_0,X_1)^2\right], \label{eq:acs}
\end{align}
where we have used the fact that $\omega^{(z)}$ is supported on edges $e$ such that $e \subseteq \rball(z,2 R^{1+\e})$.

Apply the biased Mass-Transport Principle \eqref{eq:masstp} with the functional
\[
   F(G,y,z,U_{R,R'}(\lambda)) = \con^G_y \left|\rball(y,2 R^{1+\e}) \cap U_{R,R'}(\lambda)\right| \1_{U_{R,R'}(\lambda)}(z) \E[\omega^{(z)}(X_0,X_1)^2 \mid X_0=y]
\]
to conclude from \eqref{eq:acs} that
\begin{align*}
\E\left[\hat{\omega}(X_0,X_1)^2\right]
                      &\leq 
                      \E\left[\frac{\1_{\bm{U}_{R,R'}(\lambda)}(\rho)}{\con^G_{\rho}} \sum_{z \in \rball(\rho,2 R^{1+\e})} \left|\rball(z,2 R^{1+\e})\cap\bm{U}_{R,R'}(\lambda)\right| \con^G(z)
                      \E\left[ \omega^{(\rho)}(X_0,X_1)^2 \mid X_0=z\right]\right] \nonumber \\
                      &=
                      \E\left[\frac{\1_{\bm{U}_{R,R'}(\lambda)}(\rho)}{\con^G_{\rho}} \sum_{z \in \rball(\rho,2 R^{1+\e})} \left|\rball(z,2 R^{1+\e})\cap\bm{U}_{R,R'}(\lambda)\right| \sum_{y : \{y,z\} \in E(G)} \con^G(\{y,z\}) \omega^{(\rho)}(y,z)^2\right]
                      \nonumber \\
                      &\leq
                      \E\left[\frac{\1_{\bm{U}_{R,R'}(\lambda)}(\rho)}{\con^G_{\rho}} |\rball(\rho,4 R^{1+\e}) \cap \bm{U}_{R,R'}(\lambda)| \sum_{\{y,z\} \in E(G)} \con^G(\{y,z\}) \omega^{(\rho)}(y,z)^2\right]
                      \nonumber \\
                      &=
                      \E\left[\frac{\1_{\bm{U}_{R,R'}(\lambda)}(\rho)}{\con^G_{\rho}}  \left|\rball(\rho,4 R^{1+\e})\cap\bm{U}_{R,R'}(\lambda)\right| \amod^G(\rho,R,2 R^{1+\e}) \1_{\cS'(\e,R)}(\rho) \right].
\end{align*}
Now \eqref{eq:samp1}
gives, for every $x \in \rball(\rho,4R^{1+\e})$,
\[
   \Pr[x \in \bm{U}_{R,R'}(\lambda) \mid (G,\rho)] \leq 
   \frac{\lambda \con^G_{x}}{\max \{ \vol^G(y,R): y \in \rball(x,R') \}} \leq
   \frac{\lambda \con^G_{x}}{\max \{ \vol^G(y,R): y \in \rball(\rho,R) \}},
\]
where we have used $R' = 5R^{1+\e} \geq 4R^{1+\e} + R$.
Along with independence of the sampling procedure, this yields
\[
   \E\left[\1_{\bm{U}_{R,R'}(\lambda)}(\rho) \left|\rball(\rho,4 R^{1+\e})\cap\bm{U}_{R,R'}(\lambda)\right| \mid (G,\rho) \right]
   \leq \lambda^2 \con_{\rho}^G \frac{1 + \vol^G(\rho,4 R^{1+\e})}{\left(\max \{ \vol^G(y,R): y \in \rball(\rho,R) \}\right)^2}
\]
Therefore
\[
\E\left[\hat{\omega}(X_0,X_1)^2\right]
\leq \lambda^2 \E\left[\1_{\cS'(\e,R)}(\rho) \frac{1+ \vol^G(\rho,4 R^{1+\e})}{\left(\max \{  \vol^G(y,R): y \in \rball(\rho,R) \}\right)^2}
\amod^G(\rho,R,2 R^{1+\e}) \right]
\leq \lambda^2 R^{-d_*+\e},
\]
by definition of $\cS'(\e,R)$.

Let us use \eqref{eq:samp2} with $r=R-1$ to bound
\begin{align*}
   \E[\tilde{\omega}(X_0,X_1)^2] &\leq \Pr\left[d^{G}(\rho,\bm{U}_{R,R'}(\lambda)) \geq R \mid \rho \in \cS(\e,R)\right] \\
                                 &\leq \E\left[\exp\left(-\lambda \frac{\vol^G(\rho,R-1)}{\vol^G(\rho,15 R^{1+\e})}\right) \bigmid \rho \in \cS(\e,R)\right] \\
                                 &\leq \exp(-\lambda d_* R^{-2\e} \log R),
\end{align*}
where the last line follows from the definition of $\cS(\e,R)$.
Now choose $\lambda \seteq R^{2\e}$, yielding
\[
   \E\left[\omega_R(X_0,X_1)^2\right] \leq 2 \left(\E[\hat{\omega}(X_0,X_1)^2+\tilde{\omega}(X_0,X_1)^2]\right) \leq 2 R^{-d_*+4\e}.\qedhere
\]
\end{proof}

\section{Markov type and the rate of escape}
\label{sec:mt}

Our goal now is to prove \pref{thm:mt}.
It is essentially a consequence of the fact that every
$N$-point metric space has maximal Markov type $2$ with constant $O(\log N)$
(see \pref{sec:mmt} below), and that the random walk on a reversible random graph
with almost sure subexponential growth (in the sense of \eqref{eq:subexp})
can be approximated, quantitatively, by a limit of random walks
restricted to finite subgraphs.

\subsection{Restricted walks on clusters}

\begin{definition}[Restricted random walk] \label{def:restricted}
Consider a network $G=(V,E,\con^G)$ and a finite subset $S \subseteq V$.
Let \[N_G(x) = \{ y \in V : \{x,y\} \in E \}\]
denote the neighborhood of a vertex $x \in V$.

Define a measure $\pi_S$ on $S$ by
\begin{equation}\label{eq:piS}
   \pi_S(x) \seteq \frac{\con^G_x}{\con^G(E^G(S))} \1_{S}(x)\,,
\end{equation}
where $E^G(S) \seteq \left\{ \{x,y\} \in E(G) : \{x,y\} \cap S \neq \emptyset \right\}$
is the set of edges incident on $S$.

We define {\em the random walk restricted to $S$} as the following process $\{Z_t\}$: For $t \geq 0$, put
\[ 
   \Pr(Z_{t+1} = y \mid Z_t = x) = 
   \begin{cases} 
      \frac{\con^G(E^G(x, V \setminus S))}{\con^G_x} & y = x \\
      \frac{\con^G(\{x,y\})}{\con^G_x} & y \in N_G(x) \cap S\\ 
      0 & \textrm{otherwise,}
   \end{cases}
\]
where we have used the notation $E^G(x, U) \seteq \left\{ \{x,y\} \in E : y \in U \right\}$.
It is straightforward to check that $\{Z_t\}$ is a reversible Markov chain on $S$ with stationary measure $\pi_S$. If $Z_0$ has law $\pi_S$, we say that $\{Z_t\}$ is the {\em stationary random walk restricted to $S$.}
\end{definition}

A {\em bond percolation on $G$} is a mapping $\xi : E(G) \to \{0,1\}$.
For a vertex $v \in V(G)$ and a bond percolation $\xi$, we let $K_{\xi}(v)$ denote
the connected component of $v$ in the subgraph of $G$ given by $\xi^{-1}(1)$.
Say that a bond percolation $\xi : E(G) \to \{0,1\}$ is {\em finitary} if
$K_{\xi}(\rho)$ is almost surely finite.

\begin{lemma}\label{lem:nobias}
   Suppose $(G,\rho,\xi)$ is a reversible random network and $\xi$ is finitary.
   Let $\hat{\rho} \in V(G)$ be chosen according to the measure $\pi_{K_{\xi}(\rho)}$ from \pref{def:restricted}.
   Then $(G,\rho)$ and $(G,\hat{\rho})$ have the same law.
\end{lemma}

\begin{proof}
Define the transport
\[
   F(G,x,y,\xi) \seteq \con^G_x \frac{\con^G_y}{\con^G(E^G(K_{\xi}(x)))} \1_{K_{\xi}(x)}(y) \1_{\cS}(G,x),
\]
where $\cS$ denotes some Borel measurable subset of $\rgraphs$ (recall the definition from \pref{sec:mtp}).
Then the biased mass-transport principle \eqref{eq:masstp} gives
\begin{align*}
   \Pr[(G,\rho) \in S] 
   &= \E\left[\frac{1}{\con^G_{\rho}} \sum_{x \in V(G)} F(G,\rho,x,\xi)\right] \\
   &= \E\left[\frac{1}{\con^G_{\rho}} \sum_{x \in V(G)} F(G,x,\rho,\xi)\right] 
   = \E\left[\sum_{x \in K_{\xi}(\rho)} \frac{\con^G_{x}}{\con^G(E^G(K_{\xi}(\rho)))} \1_{S}(G,x)\right],
\end{align*}
and
\[
   \Pr[(G,\hat{\rho}) \in S] = \E\left[\sum_{x \in K_{\xi}(\rho)} \pi_{K_{\xi}(\rho)}(x) \1_{\cS}(G,x)\right]
   =\E\left[\sum_{x \in K_{\xi}(\rho)} \frac{\con^G_x}{\con^G(E^G(K_{\xi}(\rho)))} \1_{\cS}(G,x)\right].\qedhere
\]
\end{proof}

\subsection{Maximal Markov type}
\label{sec:mmt}

A metric space $(\cX,d_{\cX})$ has {\em maximal Markov type $2$ with constant $K$} if
it holds that for every finite state space $\Omega$, every map $f : \Omega \to X$,
and every stationary, reversible Markov chain $\{Z_n\}$ on $\Omega$,
\[
   \E\left[\max_{0 \leq t \leq n} d_{\cX}(Z_0,Z_t)^2\right] \leq K^2 n \E\left[d_{\cX}(Z_0,Z_1)^2\right],\quad \forall n \geq 1.
\]
This is a maximal variant of K. Ball's Markov type \cite{Ball92}.
Note that every Hilbert space has maximal Markov type $2$ with constant $K$
for some universal $K$ (independent of the Hilbert space); see, e.g., \cite[\S 8]{npss06}.
Bourgain's embedding theorem \cite{Bourgain85} asserts that every $N$-point metric
space embeds into a Hilbert space with bilipschitz distortion $O(\log N)$, yielding the following.

\begin{lemma}\label{lem:mmt}
   If $(\cX,d_{\cX})$ is a finite metric space with $N = |\cX|$, then for every stationary, reversible
   Markov chain $\{Z_n\}$ on $\cX$, it holds that
\[
   \E\left[\max_{0 \leq t \leq n} d_{\cX}(Z_0,Z_t)^2\right] \leq O(n) (\log N)^2 \E\left[d_{\cX}(Z_0,Z_1)^2\right],\quad \forall n \geq 1.
\]
\end{lemma}

\subsection{Proof of \pref{thm:mt}}

We are ready to prove \pref{thm:mt}.
Consider a finitary bond peroclation $\xi : E(G) \to \{0,1\}$ such that $(G,\rho,\xi)$ is a reversible random network,
and let $\{X^{\xi}_n\}$ be simple random walk
on $K_{\xi}(\rho)$, where $X^{\xi}_0$ has law $\pi_{K_{\xi}(\rho)}$.
By \pref{lem:nobias}, $(G,\rho)$ and $(G, X^{\xi}_0)$ have the same law.
Let $\{X_n\}$ denote simple random walk on $G$ started from $\rho$.

Thus there is a natural coupling of $\{X_n\}$ and $\{X_n^{\xi}\}$ such that
\begin{equation}\label{eq:good-couple}
   \{X_0,X_1,\ldots,X_{\tau_{\xi}}\} = \{X_0^{\xi}, X_1^{\xi},\ldots,X_{\tau_{\xi}}^{\xi}\},
\end{equation}
where $\tau_{\xi} \seteq \min \{ t \geq 0 : X_{t+1} \notin V(K_{\xi}(\rho)) \}$.

\begin{lemma}\label{lem:trace}
   Suppose $(G,\rho,\omega,\xi)$ is a reversible random network, where $\xi$ is a finitary bond percolation.
   Then almost surely
   \begin{align*}
      \E\left[\max_{0 \leq t \leq \tau_{\xi} \wedge n} \dist_{\omega}^G(X_0,X_t)^2 \right. &\left.\mid (G,\rho,\omega,\xi)\right] 
       \leq
      O(n) \left(\log |K_{\xi}(\rho)|\right)^2 \E\left[\omega\!\left(X_0^{\xi},X_1^{\xi}\right)^2 \bigmid (G,\rho,\omega,\xi)\right].
   \end{align*}
\end{lemma}

\begin{proof}
   Consider the stationary, reversible Markov chain $\{X_n^{\xi}\}$ on $K_{\xi}(\rho)$.
   Applying \pref{lem:mmt} to the metric space $(V(K_{\xi}(\rho)), \dist^{K_{\xi}(\rho)}_{\omega})$, we obtain that almost surely over the
   choice of $(G,\rho,\omega,\xi)$,
   \begin{align*}
      \E\left[\max_{0 \leq t \leq n} \dist_{\omega}^{K_{\xi}(\rho)}(X^{\xi}_0,X^{\xi}_n)^2 \mid (G,\rho,\omega,\xi)\right]
      &\leq O(n) \left(\log |K_{\xi}(\rho)|\right)^2 \E\left[\omega(X_0^{\xi},X_1^{\xi})^2 \mid (G,\rho,\omega,\xi)\right] \\
      &\leq O(n) \left(\log |K_{\xi}(\rho)|\right)^2 \E\left[\omega(X_0^{\xi},X_1^{\xi})^2 \mid (G,\rho,\omega,\xi)\right].
   \end{align*}

   To conclude, we use the coupling that gives \eqref{eq:good-couple}, along with 
   the fact that $\dist^G_{\omega} \leq \dist^{K_{\xi}(\rho)}_{\omega}$
   for all $x,y \in V(K_{\xi}(\rho))$
   to arrive at
   \begin{equation*}\label{eq:pen3}
      \E\left[\max_{0 \leq t \leq \tau_{\xi} \wedge n} \dist_{\omega}^G(X_0,X_t)^2 \mid (G,\rho,\omega,\xi)\right] 
      \leq
      \E\left[\max_{0 \leq t \leq n} \dist_{\omega}^{K_{\xi}(\rho)}(X^{\xi}_0,X^{\xi}_t)^2 \mid (G,\rho,\omega,\xi)\right].\qedhere
   \end{equation*}
\end{proof}

We need a unimodular random partitioning scheme that adapts to the volume measure.
Here we state it for any unimodular vertex measure.
This argument employs a unimodular variation on the method and analysis from \cite{CKR01}, adapted
to an arbitrary underlying measure as in \cite{klmn05}.
We will use the notation $\diam^G(S) \seteq \max \{ \dist^G(x,y) : x,y \in S \}$.

\begin{lemma}\label{lem:bond}
   Suppose $(G,\rho,\mu)$ is a reversible random network, 
   where $\mu : V(G) \to \R_+$ satisfies $\mu(\rho) > 0$
   almost surely.
   Then for every $\Delta > 0$,
   there is a bond percolation $\xi_{\Delta} : E(G) \to \{0,1\}$ such that
   \begin{enumerate}
      \item $(G,\rho,\xi_\Delta)$ is a reversible random network.
      \item Almost surely $\diam^G(K_{\xi_{\Delta}}(\rho)) \leq \Delta$.
      \item For every $r \geq 0$, it holds that almost surely
         \[
            \Pr\left[B^G(\rho,r) \nsubseteq K_{\xi_{\Delta}}(\rho) \mid (G,\rho)\right] \leq 
            \frac{16 r}{\Delta} \left(1+\log \left(\frac{\mu\left(B^G(\rho,\frac{5}{8} \Delta)\right)}{\mu\left(B^G(\rho,\frac{1}{8} \Delta\right)}\right)\right),
         \]
   \end{enumerate}
   where we use the notation $\mu(S) \seteq \sum_{x \in S} \mu(x)$ for $S \subseteq V(G)$.
\end{lemma}

\begin{proof}
   By assumption, $G$ is locally finite, hence $B^G(\rho,\Delta)$ is finite.
   Thus we may assume that $\mu(x) > 0$ for all $x \in V(G)$ as follows:  Define
   $\hat{\mu}(x) = \mu(x)$ if $\mu(x) > 0$ and $\hat{\mu}(x)=1$ otherwise.
   We may then prove the lemma for $\hat{\mu}$, and observe that
   because properties (2) and (3) only refer to finite neighborhoods of the root,
   $\mu$ and $\hat{\mu}$ are identical on these neighborhoods, except for a set of zero measure.

   Let $\{\beta_x : x \in V(G)\}$ be a sequence of independent random variables 
   where $\beta_x$ is an exponential with rate $\mu(x)$.
   Let $R \in [\frac{\Delta}{4},\frac{\Delta}{2})$ be independent and chosen uniformly random.
   For a finite subset $S \subseteq V(G)$, write $\mu(S) \seteq \sum_{x \in S} \mu(x)$.
   We need the following elementary lemma.
   \begin{lemma}\label{lem:Exp}
      For any finite subset $S \subseteq V(G)$, it holds that
      \[
         \Pr\left[\beta_x = \min \{ \beta_v : v \in S \} \mid (G,\mu)\right] = \frac{\mu(x)}{\mu(S)},\qquad \forall x \in S.
      \]
   \end{lemma}

   \begin{proof}
      A straightforward calculation shows that $\min \{ \beta_v : v \in S \setminus \{x\}\}$ is exponential with rate $\mu(S \setminus \{x\})$.
      Moreover, if $\beta$ and $\beta'$ are independent exponentials with rates $\lambda$ and $\lambda'$, respectively, then
      \[
         \Pr[\beta=\min(\beta,\beta')]=\frac{\lambda}{\lambda+\lambda'}.\qedhere
      \]
   \end{proof}

   Define a labeling $\ell : V(G) \to V(G)$, where $\ell(x) \in B^G(x,R)$ is such that
   \[
      \beta_{\ell(x)} = \min \left\{ \beta_y : y \in B^G(x,R) \right\}.
   \]
   Define the bond percolation $\xi_{\Delta}$ by
   \[
      \xi_{\Delta}(\{x,y\}) \seteq \1_{\{\ell(x)=\ell(y)\}}, \quad\{x,y\} \in E(G).
   \]
   In other words, we remove edges whose endpoints receive different labels.

   Since the law of $\xi_{\Delta}$ does not depend on $\rho$, it follows that $(G,\rho,\xi_{\Delta})$ is a reversible random network,
   yielding claim (1).
   Moreover, since $\ell(x)=z$ implies that $\dist^G(x,z) \leq R \leq \Delta$, it holds that almost surely
   \[
      \diam^G(K_{\xi_{\Delta}}(\rho)) = \diam^G(\ell^{-1}(\ell(\rho))) \leq \Delta,
   \]
   yielding claim (2).

   Since the statement of the lemma is vacuous for $r > \Delta/8$, consider some $r \in [0,\Delta/8]$.
   Let $x^* \in B^G(\rho,r+R)$ be such that 
   \[
      \beta_{x^*} = \min \left\{ \beta_x : x \in B^G(\rho,R+r) \right\}.
   \]
   Then we have
   \begin{equation}\label{eq:bad}
      \Pr\left[B^G(\rho, r) \nsubseteq K_{\xi_{\Delta}}(\rho)\right] \leq \Pr\left[\dist^G(\rho,x^*) \geq R - r\right].
   \end{equation}
   
   For $x \in B^G(\rho,2\Delta)$, define the interval $I(x) \seteq [\dist^G(\rho,x) - r, \dist^G(\rho,x)+r]$.
   Note that the bad event $\{ \dist^G(\rho,x^*) \geq R - r \}$ coincides with the event $\{ R \in I(x^*) \}$.
   Order the points of $B^G(\rho,2\Delta)$ in non-decreasing order from $\rho$: $x_0 = \rho, x_1,x_2,\ldots,x_N$.
   Then \eqref{eq:bad} yields
   \begin{align}
      \Pr\left[B^G(\rho,r) \nsubseteq K_{\xi_{\Delta}}(\rho)\right]
      &\leq \Pr[R \in I(x^*)] \nonumber \\
      &=\sum_{j=1}^N \Pr[R \in I(x_j)] \cdot \Pr[x_j = x^* \mid R \in I(x_j)] \nonumber \\
      &\leq \frac{2r}{\Delta/8}\sum_{j=1}^N \Pr[x_j = x^* \mid R \in I(x_j)].\label{eq:ckr}
   \end{align}
   Note that since $R \geq \Delta/4$ and $r \leq \Delta/8$,
   \[
      R \in I(x_j) \implies x_j \in B^G(\rho,\tfrac{5}{8}\Delta) \setminus B^G(\rho,\tfrac{1}{8} \Delta).
   \]
   Observe, moreover, that $R \in I(x_j)$ implies $x_1,x_2,\ldots,x_j \in B^G(\rho,R+r)$, 
   hence
   \[
      \Pr[x_j=x^* \mid R \in I(x_j)] =
      \Pr\left[\beta_{x_j} = \min \left\{\beta_x : x \in B^G(\rho,R+r)\right\} \mid R \in I(x_j)\right] \leq \frac{\mu(x_j)}{\mu(\{x_1,x_2,\ldots,x_j\})},
   \]
   where the last inequality follows from \pref{lem:Exp}.

   Plugging these bounds into \eqref{eq:ckr} gives
   \begin{align*}
      \Pr\left[B^G(\rho,r) \nsubseteq K_{\xi_{\Delta}}(\rho)\right] &\leq \frac{16r}{\Delta} \sum_{j=|B^G(\rho,\frac{1}{8} \Delta)|+1}^{|B^G(\rho, \frac{5}{8} \Delta)|} \frac{\mu(x_j)}{\mu(x_1)+\cdots+\mu(x_j)}.
   \end{align*}
   Finally, observe that for any $a_0,a_1,a_2,\ldots,a_m > 0$,
   \begin{align*}
      \sum_{j=1}^m \frac{a_j}{a_0+a_1+a_2+\cdots+a_j} &= 
      \sum_{j=1}^m \frac{a_j/a_0}{1+a_1/a_0+\cdots+a_j/a_0} \\
                                                  &\leq \int_{0}^{(a_1+\cdots+a_j)/a_0} \frac{1}{t+1}\,dt = \log \left(1+\frac{a_1+\cdots+a_j}{a_0}\right),
   \end{align*}
   and therefore
   \[
      \Pr\left[B^G(\rho,r) \nsubseteq K_{\xi_{\Delta}}(\rho)\right] \leq \frac{16r}{\Delta} 
      \log\left(1 + \frac{\mu(B^G(\rho,\tfrac58 \Delta))}{\mu(B^G(\rho,\tfrac18 \Delta))}\right),
   \]
   as desired (noting that $\log(1+y) \leq 1+\log(y)$ for $y \geq 1$).
\end{proof}

With this in hand, we can proceed to our goal of proving \pref{thm:mt}.

\begin{proof}[Proof of \pref{thm:mt}]
   Recall that $(G,\rho,\omega)$ is a reversible random network and $\{X_n\}$ is the random walk on $G$
   started from $X_0=\rho$.  %

  Define the random vertex measure $\mu(x) \seteq c^G_x$ for $x \in V(G)$.
  For $k \geq 1$, denote $j_k \seteq 4^k$, and
   let $(G,\rho,\omega,\langle \xi_{j_k} : k \geq 1\rangle)$ be the reversible random network
   provided by applying \pref{lem:bond} with $\mu$ and $\Delta=4^k$ for each $k \geq 1$.
   Denote
   \[
      \bm{n}_k \seteq \frac{j_k}{16 k^2 \left(1+\log \frac{\vol^G(\rho,j_k)}{c_{\rho}^G}\right)}
   \] 
   so that for $k \geq 1$, almost surely,
   \begin{gather}
      V(K_{\xi_{j_k}}(\rho)) \subseteq B^G(\rho,j_k)\,,\label{eq:diambound} \\
      \Pr[B^G(\rho,\bm{n}_k) \nsubseteq V(K_{\xi_{j_k}}(\rho)) \mid (G,\rho,\omega)] \leq O(k^{-2})\,.\label{eq:ballprob}
   \end{gather}
         By the Borel-Cantelli lemma, it holds thats almost surely over the choice of $(G,\rho,\omega)$,
         \[
            \# \left\{ k \geq 1 : B^G(\rho,\bm{n}_k) \nsubseteq K_{\xi_{j_k}}(\rho) \right\} < \infty.
         \]

   For $\delta > 0$, denote $V_{\delta} \seteq \{ x \in V(G) : c^G_x > \delta \}$, and observe that
   Markov's inequality yields
   $\Pr[X_0 \in V_{\delta}] \leq C_0 \delta$, where $C_0 \seteq \E[1/c_{\rho}^G]$, and we recall
   that $C_0 < \infty$ by assumption.
   Therefore since $(G,X_0)$ and $(G,X_n)$ have the same law for every $n \geq 0$,
   a union bound gives that for any $n \geq 0$,
   \begin{equation}
      \label{eq:borcan1}
      \Pr\left[X_0,X_1,\ldots,X_{n} \in V_{\delta}\right] \geq 1 - \delta (n+1) C_0.
   \end{equation}
   Denote the event $\cE_k \seteq \left\{X_0,X_1,\ldots,X_{\bm{n}_k} \in V_{1/(k^2 (\bm{n}_{k}+1))}\right\}$.
   Then since $\Pr(\cE_k) \geq 1 - C_0 k^{-2}$, the Borel-Cantelli lemma implies that almost surely,
   \begin{equation}\label{eq:borcan2}
      \# \left\{ k \geq 1 : X_0,X_1,\ldots,X_{\bm{n}_k} \nsubseteq V_{1/(k^2 (\bm{n}_{k}+1))} \right\} < \infty
   \end{equation}

   Define now the bond percolations $\langle \tilde{\xi}_k,\hat{\xi}_k : k \geq 1\rangle$ by
   \[
      \tilde{\xi}_k(\{x,y\})=1 \iff c^G_x \wedge c^G_y \leq \frac{1}{k^2 (\bm{n}_k+1)},
   \]
   and
   \[
      \hat{\xi}_k \seteq \xi_{j_k} \tilde{\xi}_k\,.
   \]
   Note that $B^G(\rho,\bm{n}_k) \subseteq K_{\xi_{j_k}}(\rho)$ and $X_0,X_1,\ldots,X_n \in V_{1/(k^2 (\bm{n}_k+1))}$
   imply that
   \[
      \left\{X_0,X_1,\ldots,X_{\bm{n}_k}\right\} = \left\{X_0^{\hat{\xi}_k}, X_1^{\hat{\xi}_k},\ldots,X_{\bm{n}_k}^{\hat{\xi}_k} \right\}
   \]
   under the coupling described in \eqref{eq:good-couple}, hence
   \eqref{eq:borcan1} and \eqref{eq:borcan2} imply that almost surely
   \begin{equation}\label{eq:borcan3}
      \# \left\{ k \geq 1 : \tau_{\hat{\xi}_k} < \bm{n}_k \right\} < \infty.
   \end{equation}

   \begin{lemma}\label{lem:component-volume}
            It holds that
            \[
               \left|K_{\hat{\xi}_k}(\rho)\right|
               \leq \max\left(1,k^2 \bm{n}_k \vol^G(\rho,4^k)\right).
            \]
   \end{lemma}

   \begin{proof}
   Note that either $\{\rho\} = V(K_{\tilde{\xi}_k}(\rho))$ or otherwise $c^G_x > (k^2 \bm{n}_k)^{-1}$ for every $x \in V(K_{\tilde{\xi}_k}(\rho))$, hence
         \[
               \left|K_{\hat{\xi}_k}(\rho)\right|
             \leq \max\left(1,k^2 \bm{n}_k \mu\left(K_{\hat{\xi}_{k}}(\rho)\right)\right)
            \leq\max\left(1,k^2 \bm{n}_k \mu\left(K_{{\xi}_{j_k}}(\rho)\right)\right).
         \]
         Now the desired inequality follows from \eqref{eq:diambound} and $j_k = 4^k$.
   \end{proof}

   Using \eqref{eq:borcan3} and \pref{lem:component-volume} in conjunction with \pref{lem:trace} now gives:
   Almost surely, for all but finitely many $k$,
    \begin{equation*}\label{eq:pen2}
      \E\left[\max_{0 \leq t \leq \bm{n}_k} \dist_{\omega}^G(X_0,X_t)^2 \mid (G,\rho,\omega)\right] \leq
      O(n) \left(\log \left(1+\vol^G(\rho,4^k)\right)\right)^2 \E\left[\omega(X_0^{\hat{\xi}_{k}},X_1^{\hat{\xi}_{k}})^2 \mid (G,\rho,\omega)\right].
   \end{equation*}
   Combining this with \eqref{eq:subexp} yields
   that almost surely,
   \begin{equation}\label{eq:pen}
      \lim_{n \to \infty}
      \frac{\log \E\left[\max_{0 \leq t \leq n} \dist_{\omega}^G(X_0,X_t)^2 \mid (G,\rho)\right]}{\log n}
      \leq 1 + \lim_{k \to \infty} \frac{\log \E[\omega(X_0^{\hat{\xi}_{k}},X_1^{\hat{\xi}_{k}})^2 \mid (G,\rho)]}{\log n}.
   \end{equation}
   Recalling \pref{lem:nobias}, for every $k \geq 1$ we have $\E[\omega(X_0^{\hat{\xi}_k},X_1^{\hat{\xi}_k})^2] \leq \E[\omega(X_0,X_1)^2] < \infty$,
   and therefore
   \[
      \Pr\left(\E[\omega(X_0^{\hat{\xi}_k},X_1^{\hat{\xi}_k})^2 \mid (G,\rho)] > k^2 \E[\omega(\rho)^2]\right) \leq k^{-2},
   \]
   so again Borel-Cantelli tells us that almost surely
   \[
     \# \left\{ k \geq 1 : \E\left[\omega(X_0^{\hat{\xi}_k},X_1^{\hat{\xi}_k})^2\right] > k^2 \E\left[\omega(X_0,X_1)^2\right] \right\} < \infty
   \]
   Plugging this into \eqref{eq:pen} yields that almost surely
   \begin{equation*}\label{eq:conclusion}
      \lim_{n \to \infty}
      \frac{\log \E\left[\max_{0 \leq t \leq n} \dist_{\omega}^G(X_0,X_t)^2 \mid (G,\rho,\omega)\right]}{\log n} \leq 1,
   \end{equation*}
   completing the proof.
\end{proof}

\section{Exponent relations}

Let us now prove the nontrivial inequalities in \pref{thm:inequalities}.

\subsection{The speed upper bound}
\label{sec:beta}

\begin{theorem}
   If $(G,\rho)$ is a reversible random network satisfying $\E[1/c_{\rho}^G] < \infty$,
   then $\down{\beta}^{\cA} \geq 2 \down{d}_f - \up{d}_f + \tilde{\zeta}$.
\end{theorem}

\begin{proof}
   Recall that $\{X_n\}$ is the random walk on $G$ (cf. \pref{eq:transition}) started from $X_0=\rho$.
   Let us denote $d_* \seteq 2 \down{d}_f - \up{d}_f + \tilde{\zeta}$.
   If $d_* \leq 2$, we can use the weight $\omega \equiv 1$ for which $\dist_{\omega}^G = d^{G}$,
   and \eqref{eq:as-diffusive} yields $\down{\beta}^{\cA} \geq 2$.
   Consider now $d_* > 2$ and fix $\delta \in (0,d^*-2)$.
   Apply \pref{thm:good-weight} to arrive at a reversible
   random weight $\omega : E(G) \to \R_+$ such that $\E[\omega(X_0,X_1)^2] < \infty$
   and almost surely eventually,
   \begin{equation}\label{eq:gw}
      \dist_{\omega}^G(\rho,\crball(\rho,R)) \geq R^{(d_*-\delta)/2}.
   \end{equation}

   Now, since $\up{d}_f < \infty$, it follows that \eqref{eq:subexp} holds, and we can
   apply \pref{thm:mt} to $(G,\rho,\omega)$ yielding: Almost surely eventually (with respect to $n$),
   \[
      \E\left[\max_{0 \leq t \leq n} \dist_{\omega}^G(X_0,X_t)^2 \mid (G,\rho,\omega)\right] \leq n^{1+\delta}.
   \]
   Combining this with \eqref{eq:gw} yields almost surely eventually
   \[
      \E\left[\max_{0 \leq t \leq n} d^{G}(X_0,X_t)^{d_*-\delta} \mid (G,\rho,\omega)\right] \leq n^{1+\delta}.
   \]
   Now since $d_* - \delta > 2$, convexity of $y \mapsto y^{(d_*-\delta)/2}$ gives
   \[
      \E\left[\max_{0 \leq t \leq n} d^{G}(X_0,X_t)^{2} \mid (G,\rho,\omega)\right] \leq n^{2(1+\delta)/(d_*-\delta)}.
   \]
   Since we can take $\delta > 0$ arbitrarily small, this yields $\down{\beta}^{\cA} \geq d_*$, completing the proof.
\end{proof}

\subsection{Effective resistance and the Green kernel}
\label{sec:resistance}

Assume again that $(G,\rho)$ is a reversible random network.

\begin{definition}[Green kernels]
   For $S \subseteq V(G)$, let $\tau_S \seteq \min \{ n \geq 0 : X_n \in S \}$, and define
   the Green kernel killed off $S$ by
   \[
      \green_S^G(x,y) \seteq \E\left[\sum_{t < \tau_{V(G) \setminus S}} \1_{\{X_t=y\}} \bigmid X_0 = x \right].
   \]
   For $n \geq 1$, define
   \[
      \Green_n^G(x,y) \seteq \E\left[\sum_{t \leq n} \1_{\{X_t=y\}} \bigmid X_0 = x \right].
   \]
\end{definition}

It is well-known (see \cite[Ch. 2]{lp:book}) that for any $x \in V(G)$ and $S \subseteq V(G)$:
\begin{equation}\label{eq:basic-green}
   \con^G_x\,\reff^G(x \leftrightarrow V(G) \setminus S) =\green_{S}^G(x,x)\,.
\end{equation}
We recall the standard relationship between effective resistances and commute times \cite{CRRST96} gives
the following.

\begin{lemma}\label{lem:green}
   For any $R \geq 1$, almost surely:
   \[
      \E[\sigma_R \mid (G,\rho), X_0 = \rho] \leq \reff^G(\rho \leftrightarrow \crball(\rho,R))\, \vol^G(\rho,R).
   \]
\end{lemma}

This immediately yields \eqref{eq:relation-commute}:

\begin{theorem}\label{thm:relation-commute}
   It holds that $\up{d}^{\cA}_w \leq \up{d}_f + \tilde{\zeta}_0$.
\end{theorem}

Similarly standard arguments yields the upper and lower bounds in \eqref{eq:relation-ds}, as follows.

\begin{theorem}\label{thm:relation-ds-lb}
   It holds that
   \[
      \down{d}_s \geq 2 \left(1-\frac{\tilde{\zeta}_0}{\down{d}_w}\right)
   \]
\end{theorem}

\begin{proof}
   Since the even return times are non-increasing (see, e.g., \cite[Prop. 10.18]{lpw09}),
   we have
   \begin{equation}\label{eq:return-green}
      p_{2n}^G(\rho,\rho) \leq \frac{1}{n} \sum_{j=1}^n p_{2j}^G(\rho,\rho) \leq \frac{1}{n} \Green_{2n}^G(\rho,\rho).
   \end{equation}
   By definition, for any $\delta > 0$, we have that almost surely eventually
   \[
      \sigma_R > R^{\down{d}_w-\delta}.
   \]
   Therefore almost surely eventually
   \[
      \Green^G_{n}(\rho,\rho) \leq \green^G_{\rball\left(\rho,n^{1/(\down{d}_w-\delta)}\right)}(\rho,\rho)
      \stackrel{\eqref{eq:basic-green}}{=} \con^G_{\rho}\,\reff^G\left(\rho \leftrightarrow \crball(\rho,n^{1/(\down{d}_w-\delta)})\right)
      \leq \con^G_{\rho}\,n^{(\tilde{\zeta}_0+\delta)/(\down{d}_w-\delta)}.
   \]
   Combined with \eqref{eq:return-green},
   this gives almost surely eventually
   \[
      p_{2n}^G(\rho,\rho) \leq 2 \con^G_{\rho} (2n)^{(\tilde{\zeta}_0+\delta)/(\down{d}_w-\delta)-1},
   \]
   and since this holds for all $\delta > 0$, we obtain $\down{d}_s \geq 2 (1-\tilde{\zeta}_0/\down{d}_w)$.
\end{proof}

\begin{theorem}\label{thm:relation-ds-ub}
   It holds that
   \[
      \up{d}_s \leq \frac{2 \up{d}_f}{\down{d}_w}.
   \]
\end{theorem}

\begin{proof}
   Using reversibility, we have almost surely
   \begin{align*}
      p^G_{2n}(\rho,\rho) \geq \sum_{x \in\rball(\rho,R)} p_n^G(\rho,x) p_n^G(x,\rho) 
      = \con^G_{\rho} \sum_{x \in \rball(\rho,R)} \frac{p_n^G(\rho,x)^2}{\con^G_x} .
   \end{align*}
   Thus applying Cauchy-Schwarz yields
   \begin{align}
      \frac{p^G_{2n}(\rho,\rho)}{\con^G_{\rho}} \geq \frac{\left(\sum_{x \in \rball(\rho,R)} p^G_n(\rho,x)\right)^2}{\vol^G(\rho,R)}
      \geq \frac{\left(\Pr[X_n \in \rball(\rho,R) \mid (G,\rho)]\right)^2}{\vol^G(\rho,R)}. \label{eq:cs}
   \end{align}
   Observe that
   \begin{equation}\label{eq:time1}
      \Pr[X_n \in \rball(\rho,R) \mid (G,\rho)] \geq \Pr[\sigma_R \geq n \mid (G,\rho)].
   \end{equation}
   By definition, for every $\delta > 0$, almost surely eventually
   $\sigma_R \geq R^{\down{d}_w-\delta}$ and $\vol^G(\rho,R) \leq R^{\up{d}_f+\delta}$.
   Combining these with \eqref{eq:cs} and \eqref{eq:time1} gives
   almost surely eventually
   \[
      \frac{ p_{2n}^G(\rho,\rho)}{\con^G_{\rho}} \geq \left(\vol^G\left(\rho,n^{1/(\down{d}_w-\delta)}\right)\right)^{-1} \geq n^{-(\up{d}_f+\delta)/(\down{d}_w-\delta)}.
   \]
   As this holds for every $\delta > 0$, it yields the claimed inequality.
\end{proof}

Finally, let us prove that the assumptions \eqref{eq:sr-ub} and \eqref{eq:sr-lb} imply $\tilde{\zeta}=\tilde{\zeta}_0$ in the case $\zeta > 0$.
The first part of the argument follows \cite[\S 3.2]{bck05}.
The second part uses methods similar to those employed by Telcs \cite{Telcs89}.

\begin{theorem}\label{thm:sr-compare}
   If \eqref{eq:sr-ub} and \eqref{eq:sr-lb} hold for some $\zeta > 0$, then $\tilde{\zeta}=\tilde{\zeta}_0=\zeta$.
\end{theorem}

\begin{proof}
   First note that if $d^{G}(\rho,x)=R+1$, then
   \begin{equation}\label{eq:fp}
      \reff^G(\rho \leftrightarrow \crball(\rho,R)) \leq \reff^G(\rho \leftrightarrow x) \stackrel{\eqref{eq:sr-ub}}{\leq} (R+1)^{\tilde{\zeta}+\delta},
   \end{equation}
   hence \eqref{eq:sr-ub} yields
   \begin{equation}\label{eq:sr-e1}
      \tilde{\zeta}_0 \leq \zeta.
   \end{equation}
   Thus we are left to prove that $\tilde{\zeta} \geq \zeta$.

   For $y \in V(G)$ and $R \geq 1$, define
   \begin{equation}\label{eq:QG}
      Q^R_{\rho}(y) \seteq \Pr\left[\tau_{\{\rho\}} < \tau_{\crball(\rho,R)} \mid X_0=y\right] = \frac{\con^G_{\rho}}{\con^G_y} \frac{\green_{\rball(\rho,R)}(\rho,y)}{\green_{\rball(\rho,R)}(\rho,\rho)},
   \end{equation}
   where the latter equality arises because 
   both $Q^R_{\rho}$ and the function $y \mapsto \green_{\rball(\rho,R)}(\rho,y)/\con^G_y$ are harmonic on $\rball(\rho,R) \setminus \{\rho\}$.
   Moreover, $Q^R_{\rho}$ and the right-hand side
   vanish on $\crball(\rho,R)$ and are equal to $1$ at $\rho$.

   Hence, the Dirichlet principle \eqref{eq:dirichlet} yields
   \begin{equation}\label{eq:energyQ}
      \energy^G(Q_{\rho}^R) = \frac{1}{\reff^G(\rho \leftrightarrow \crball(\rho,R))}.
   \end{equation}
   In particular, we have
   \begin{equation}\label{eq:oneps}
      \left|1-Q^R_{\rho}(y)\right|^2 = \left|Q^R_{\rho}(\rho)-Q^R_{\rho}(y)\right|^2 \leq \reff^G(\rho \lra y)\,\energy^G(Q^R_{\rho}) = \frac{\reff^G(\rho \lra y)}{\reff^G(\rho \lra \crball(\rho,R))},
   \end{equation}
   where the inequality is another application of the Dirchlet principle \eqref{eq:dirichlet}.

   Assume now that $\zeta > 0$,
   and fix $\delta \in (0,\zeta)$.
   Denote $R' \seteq R^{(\zeta+2\delta)/(\zeta-\delta)}$ and
   $Q_{\rho} \seteq Q_{\rho}^{R'}$.
   Using \eqref{eq:sr-ub} and \eqref{eq:sr-lb}, we have almost surely eventually
   \begin{align}
      \max \{ \reff^G(\rho \lra x) : x \in \rball(\rho,R) \} &\leq R^{\zeta+\delta}, \label{eq:t1} \\
      \reff^G\left(\rho \leftrightarrow \crball(\rho,R')\right) &\geq R^{\zeta+2\delta}.\label{eq:t2}
   \end{align}
   So by \eqref{eq:oneps},
   almost surely eventually
   \begin{equation}\label{eq:half}
      \min \left\{ Q_{\rho}(y) : y \in \rball(\rho,R) \right\} \geq 1-R^{-\delta/2} > \frac12.
   \end{equation}
   
   \begin{remark}\label{rem:zeta0-again}
      Here one notes that this conclusion cannot be reached for $\zeta=0$ because we cannot
      choose $R'$ large enough with respect to $R$ so as to create a gap between
      the respective upper and lower bounds in \eqref{eq:t1} and \eqref{eq:t2}.
      Indeed, it is this sort of gap that Telcs defines as ``strongly recurrent''
      (see \cite[Def. 2.1]{Telcs01}), although his quantitative notion (which requires 
      a uniform multiplicative gap with $R' = O(R)$) is too strong for us, as it entails $\tilde{\zeta} > 0$.
   \end{remark}

   Let us assume that $R$ is such that \eqref{eq:half} holds.
   Denote by $H$ the induced graph on $G[\rball(\rho, 2 R')]$, and consider the sets
   \begin{align*}
      V_{1/2} &\seteq \left\{ x \in \rball(\rho,R') : Q_{\rho}(x) = 1/2 \right\}, \\
      E_{1/2} &\seteq \left\{ \{x,y\} \in E(H) : Q_{\rho}(x) < 1/2 \leq Q_{\rho}(y) \right\}.
   \end{align*}
   Define a new graph $\tilde{H}$ where each edge $e = \{x,y\} \in E_{1/2}$ is replaced by
   a pair of edges
   $e_x = \{x,v_{xy}\}, e_y = \{v_{xy},y\}$ with
   conductances satisfying the system
   \begin{align}
      \frac{1}{\con^{\tilde{H}}(e_x)} + \frac{1}{\con^{\tilde{H}}(e_y)} &= \frac{1}{\con^H(e)}\,,\label{eq:effcon} \\
      \frac{\con^{\tilde{H}}(e_x)+\con^{\tilde{H}}(e_y)}{2} &= \con^{\tilde{H}}(e_x) Q_{\rho}(x) + \con^{\tilde{H}}(e_y) Q_{\rho}(y) \nonumber
   \end{align}
   and $\con^{\tilde{H}}(e)=\con^H(e)$ for the remaining original edges $\{e \in E(\tilde{H}) : e \subseteq V(H) \}$.

   Denote $\tilde{V}_{1/2} \seteq V_{1/2} \cup \{ v_{xy} : \{x,y\} \in E_{1/2} \}$, and
   extend $Q_{\rho}$ to the new vertices so that $\tilde{Q}_{\rho}(v) = 1/2$ for $v \in \tilde{V}_{1/2}$.
   Then:
   \begin{enumerate}
      \item $\tilde{Q}_{\rho}(\rho)=1$, $\tilde{Q}_{\rho}$ is harmonic on $(\rball(\rho,R') \cup \tilde{V}_{1/2}) \setminus \{\rho\}$,
      \item $\tilde{Q}_{\rho}$ vanishes elsewhere on $V(\tilde{H})$, and
      \item $\energy^{\tilde{H}}(\tilde{Q}_{\rho}) = \energy^{H}(Q_{\rho})$.
   \end{enumerate}
   Since $\tilde{Q}_{\rho}(\tilde{V}_{1/2}) = 1/2$ and $\tilde{Q}_{\rho}(\crball(\rho,R'))=0$,
   we conclude from the Dirichlet principle and (1)--(3) that
   \[
      \reff^{\tilde{H}}(\tilde{V}_{1/2} \lra \crball(\rho,R')) = \frac{1}{4 \energy^H(Q_{\rho})} = \frac{1}{4 \energy^G(Q_{\rho})}
      = \frac{\reff^G(\rho \lra \crball(\rho,R))}{4},
   \]
   where the last equality is \eqref{eq:energyQ}.
   Moreover, by \eqref{eq:half}, it holds that $\tilde{V}_{1/2}$ separates $\rball(\rho,R)$ from $\crball(\rho,R')$ in $\tilde{H}$, and thus
   \begin{align*}
      \reff^{\tilde{H}}(\rball(\rho,R) \lra \crball(\rho,R')) 
                                                       &\geq \reff^{\tilde{H}}(\tilde{V}_{1/2} \lra \crball(\rho,R')) \geq
                                                       \frac14 \reff^G(\rho \lra \crball(\rho,R)) \geq \frac14 R^{\zeta-\delta},
   \end{align*}
   where the last inequality follows from \eqref{eq:sr-lb} and holds almost surely eventually.
   Finally, observe that by the series law for conductances, \eqref{eq:effcon} does not change the 
   effective conductance across subdivided edges, hence
   \[
      \reff^{G}(\rball(\rho,R) \lra \crball(\rho,R'))  =
      \reff^{\tilde{H}}(\rball(\rho,R) \lra \crball(\rho,R'))
                                                       \geq \frac14 R^{\zeta-\delta}.
   \]
   Since this holds for any $\delta > 0$, we conclude that $\tilde{\zeta} \geq \zeta$, as required.
\end{proof}

\subsection{Resistance exponent for planar maps coupled to a mated-CRT}
\label{sec:LQG}

We first establish that $\tilde{\zeta}=0$ for the $\gamma$-mated-CRT with $\gamma \in (0,2)$.
It is known that $\tilde{\zeta}_0 = 0$ \cite[Prop. 3.1]{GM21}.
While the following argument is somewhat technical and, to our knowledge,
does not appear elsewhere, we stress that it is an easy consequence
of \cite{GMS19,DG20}.

\smallskip

Fix some $\gamma \in (0,2)$ and for $\e > 0$, let $\cG^{\e}$ be the $\gamma$-mated-CRT with increment $\e$.
See, for instance, the description in \cite{GMS19}.
For our purposes, we may consider this as a random planar multigraph.
When needed, we can replace multiple edges by appropriate conductances.

From \cite[Thm. 1.9]{DMS14},
one can identify $V(\cG^{\e})=\e \Z$ and there is a space-filling SLE curve $\eta : \R \to \C$
parameterized by the LQG mass of the $\gamma$-quantum cone,
with $\eta(0)=0$ and such that $\{a,b\} \in E(\cG^{\e})$ are connected by an edge if and only if the
corresponding cells $\eta([a-\e,a])$ and $\eta([b-\e,b])$ share a non-trivial connected boundary arc.
Thus we can envision $\eta$ as an embedding of $V(\cG^{\e})$ into the complex plane, where
a vertex $v \in V(\cG^{\e})$ is sent to $\eta(v)$.
Let us denote the Euclidean ball $B^{\C}(z,r) \seteq \{ y \in \C : |y-z| \leq r \}$.

The underlying idea is simple:  We will arrange that, with high probability,
the image of a graph annulus under $\eta$ contains a Euclidean annulus $\cA$ of large width.
Then we pull back a Lipschitz test functional from $\cA$ to $\cG^{\e}$,
and use the Dirichlet principle \eqref{eq:dirichlet} to lower bound
the effective resistance across the annulus.

By \cite[Prop. 4.6]{DG20},
there is a number $d_{\gamma} > 2$ such that the following holds:
For every $\theta \in (0,1)$ and $\delta > 0$,
there is an $\alpha=\alpha(\delta,\gamma,\theta) > 0$ such that as $\e \to 0$,
\begin{align*}
\Pr\left[\eta\left(B^{\cG^{\e}}(0, \e^{-1/(d_{\gamma}+\delta)})\right) \subseteq B^{\C}(0,\theta)\right] &\geq 1 - O(\e^{\alpha}) \\
\Pr\left[\eta^{-1}\!\left(B^{\C}(0,\theta) \cap \eta(\e \Z)\right) \subseteq B^{\cG^{\e}}(0, \e^{-1/(d_{\gamma}-\delta)})\right] &\geq 1 - O(\e^{\alpha}).
\end{align*}
In particular, taking $\theta = 1/4$ and $\theta = 3/4$, respectively, yields, for some $\alpha=\alpha(\delta,\gamma) > 0$:
\begin{align}
   \Pr\left[
      \eta\left(B^{\cG^{\e}}(0, \e^{-1/(d_{\gamma}+\delta)})\right)\right. &\subseteq B^{\C}(0,1/4) \cap \eta(\e \Z) \nonumber \\
                                                                            &\left.\subseteq B^{\C}(0,3/4) \cap \eta(\e \Z) \subseteq \eta\left(B^{\cG^{\e}}(0, \e^{-1/(d_{\gamma}-\delta)})\right)\right] \geq 1 - O(\e^{\alpha})\,.\label{eq:trapped}
\end{align}

For a subset $D \subseteq \C$, denote
\begin{align*}
   \cV\cG^{\e}(D) &\seteq \{ x \in \e \Z : \eta([x-\e,x]) \cap D \neq \emptyset \},
\end{align*}
and let $\cG^{\e}(D)$ be the subgraph of $\cG^{\e}$ induced on $\cV\cG^{\e}(D)$.
For a function $f : \overline{D} \to \R$,
define $f^{\e} : \cV\cG^{\e}(D) \to \R$ by
\[
   f^{\e}(z) \seteq \begin{cases}
      f(\eta(z)) & z \in \cV\cG^{\e}(D) \setminus \cV\cG^{\e}(\partial D) \\
      \sup_{x \in \eta([z-\e,z]) \cap \partial D} f(z) & z\in \cV\cG^{\e}(\partial D).
   \end{cases}
\]
Take now $D \seteq  B^{\C}(0,1)$ and 
define $f : D\to \R$ by $f(z) \seteq \min(1, 4 \left(|z|-3/8\right)_+)$, which is a $4$-Lipschitz
function satisfying
\begin{equation}\label{eq:restrict}
   f|_{B^{\C}(0,3/8)} \equiv 0, \qquad  f|_{B^{\C}(0,1) \setminus B^{\C}(0,5/8)} \equiv 1.
\end{equation}

Let $\{f_n\}$ be a sequence of continuously differentiable, uniformly Lipschitz functions such that $f_n \to f$ uniformly on $D$.
Then we may apply \cite[Lem. 3.3]{GMS19} to each $f_n$ to obtain, for every $n \geq 1$,
\[
   \Pr\left(\energy^{\cG^{\e}(D)}(f_n^{\e}) \leq \e^{\alpha} + A \int_D |\nabla f_n(z)|^2\,dz\right) \geq 1- O(\e^{\alpha}),
\]
where $A=A(\gamma),\alpha(\gamma) > 0$.
We conclude that with probability at least $1-O(\e^{\alpha})$, the Dirichlet energy of $f_n^{\e}$ is uniformly (in $n$) bounded.
Taking $f^{\e} = \lim_{n \to \infty} f^{\e}_n$, we obtain
the following
in conjuction with \eqref{eq:trapped} and \eqref{eq:restrict}.

\begin{lemma}\label{lem:Gtest}
   For every $\gamma \in (0,2)$ and $\delta > 0$, there are numbers $\alpha,A > 0$ such that for every $\e > 0$,
   with probability at least $1-O(\e^{\alpha})$, there is a function $f^{\e} : V(\cG^{\e}) \to \R$
   such that
   \begin{enumerate}
      \item $f^{\e}$ vanishes on $B^{\cG^{\e}}(0,\e^{-1/(d_{\gamma}+\delta)})$,
      \item $f^{\e}$ is identically $1$ on $\partial_{\cG^{\e}} B^{\cG^{\e}}(0, \e^{-1/(d_{\gamma}-\delta)})$.
      \item $\energy^{\cG^{\e}}(f^{\e}) \leq A$.
   \end{enumerate}
   In particular, the Dirichlet principle \eqref{eq:dirichlet} gives, with probability at least $1-O(\e^{\alpha})$,
\[
   \reff^{\cG^{\e}}\left(\partial_{\cG^{\e}} B^{\cG^{\e}}(0,\e^{-1/(d_{\gamma}+\delta)})) \leftrightarrow \partial_{\cG^{\e}} B^{\cG^{\e}}(0,\e^{-1/(d_{\gamma}-\delta)})\right) \geq 1/A.
\]
\end{lemma}

Note that the law of $\cG^{\e}$ is independent of $\e > 0$, and therefore denoting its law by $\cG$ and taking $R \seteq 1/\e$, we arrive at the following.

\begin{corollary}\label{cor:zeta-crt}
   Let $\cG$ denote the $\gamma$-mated-CRT for $\gamma \in (0,2)$.
   Then for every $\delta > 0$,
   there are numbers $\alpha,\kappa > 0$
   such that with probability at least $1-O(R^{-\alpha})$
   \begin{equation}\label{eq:test-fun}
      \reff^{\cG}\left(\partial_{\cG} B^{\cG}(0, R) \leftrightarrow \partial_{\cG} B^{\cG}(0, R^{1+\delta})\right) \geq \kappa.
   \end{equation}
   In particular, it holds that for every $\delta > 0$, almost surely eventually
   \[
      \reff^{\cG}\left(\partial_{\cG} B^{\cG}(0, R) \leftrightarrow \partial_{\cG} B^{\cG}(0, R^{1+\delta})\right) \geq \kappa.
   \]
   Since this holds for every $\delta > 0$,
   and $(\cG,0)$ is a unimodular random network, we have $\tilde{\zeta} = 0$.
\end{corollary}

\begin{proof}
   \eqref{eq:test-fun} follows immediately from \pref{lem:Gtest}.
   The other conclusion is a standard consequence:  The Borel-Cantelli lemma
   implies that almost surely, for all but finitely many $k \in \N$, we have
   \[
      \reff^{\cG}\left(\partial_{\cG} B^{\cG}(0, 2^k) \leftrightarrow \partial_{\cG} B^{\cG}(0, 2^{(1+\delta)k})\right) \geq \kappa,
   \]
   so by the series law for effective resistances, it holds that almost surely eventually
   \[
      \reff^{\cG}\left(\partial_{\cG} B^{\cG}(0, R) \leftrightarrow \partial_{\cG} B^{\cG}(0, 2 R^{1+\delta})\right) \geq 
      \reff^{\cG}\left(\partial_{\cG} B^{\cG}(0, 2^{\lfloor \log_2 R\rfloor}) \leftrightarrow \partial_{\cG} B^{\cG}(0, 2^{\lceil \log_2 (R^{1+\delta})\rceil})\right) \geq 
      \kappa,
   \]
   and thus for any $\delta' > \delta$, almost surely eventually $\reff^{\cG}\left(\partial_{\cG} B^{\cG}(0, R) \leftrightarrow \partial_{\cG} B^{\cG}(0, R^{1+\delta'})\right) \geq \kappa$.
\end{proof}

Note that since $\tilde{\zeta}=\tilde{\zeta}_0=0$ and $d_f$ exists \cite{DG20},
it follows from \pref{thm:main} that $d_w=d_f$ and $d_s=2$.  
Both equalities were known previously:
$d_s \leq 2$ from \cite{lee17a},
$d_w \leq d_f$ and $d_s \geq 2$ from \cite{GM21}, and
and $d_w \geq d_f$ from \cite{GH20}.
Let us remark that the preceding argument requires somewhat less
detailed information about $\cG$ than that of \cite{GH20}.
In particular, bounding $\tilde{\zeta}$ only requires control of one
scale at a time.

\subsubsection{Other planar maps}

We consider now the case of random planar maps that can be appropriately coupled to a $\gamma$-mated CRT for some $\gamma \in (0,2)$;
we refer to \cite{GHS20} for a discussion of such examples, including the UIPT, and
random planar maps whose law is biased by the number of different spanning trees ($\gamma=\sqrt{2}$),
bipolar orientations $(\gamma=\sqrt{4/3})$, or Schynder woods ($\gamma=1$).

Our goal is to prove that $\tilde{\zeta}=0$ for each of these random planar maps $(M,\rho)$.
We employ the same approach as in the preceding section, arguing that
an annulus in $(M,\rho)$ can be mapped into $\cG$ so that its image contains
an annulus of large width, and that the Dirichlet energy of functionals in $\cG$ is
controlled when pulling them back to $M$.

Fix $\gamma \in (0,2)$ and let $\cG$ be the $\gamma$-mated-CRT with increment $1$.
Let $\cG_n$ be the subgraph of $\cG$ induced on the vertices $[-n,n] \cap \Z$.
Parts (1)--(3) in the following theorem are the conjunction of Lemma 1.11 and Theorem 1.9 in \cite{GHS20}.
Part (4) is \cite[Lem. 4.3]{GM21}.

\begin{theorem}\label{thm:rough}
   For each model considered in \cite{GHS20}, the following holds.
   There is a coupling of $(M,\rho)$ and $(\cG,0)$, and
   a family of random rooted graphs $\{ (M_{n},\rho_n) : n \geq 1\}$ and numbers
   $\alpha, K, q > 0$
   such that for every $n \geq 1$, with probability at least $1-O(n^{-\alpha})$:
   \begin{enumerate}
      \item $B^{\cG}(0,n^{1/K}) \subseteq V(\cG_{n})$,
      \item The induced, rooted subnetworks $B^{M}(\rho,n^{1/K})$ and $B^{M_{n}}(\rho_{n},n^{1/K})$ are isomorphic.
      \item There is a mapping $\phi_n : V(M_n) \to V(\cG_n)$ with $\phi_n(\rho_n)=0$, and for all $3 \leq r \leq R$,
         \begin{align*}
            \phi_n\left(B^{M_n}\!\left(\rho_n, (K \log n)^{-q} (r-2)\right)\right) &\subseteq B^{\cG_n}(0,r) \\
            \phi_n\left(V(M_n) \setminus B^{M_n}\!\left(\rho_n, (K \log n)^{q} R - 1\right)\right) &\subseteq V(\cG_n) \setminus B^{\cG_n}(0,R).
         \end{align*}
      \item For every $f : V(\cG_n) \to \R$, it holds that
         \[
            \energy^{M_n}(f \circ \phi_n) \leq K (\log n)^q \energy^{\cG_n}(f).
         \]
   \end{enumerate}
\end{theorem}

\begin{corollary}\label{cor:penpen}
   For any model considered in \cite{GHS20}, it holds that $\tilde{\zeta}=0$.
\end{corollary}

We prove this momentarily, but first note the following consequence.
Since $d_f > 2$ for each of these models \cite[Prop. 4.7]{DG20},
and $\tilde{\zeta}_0=0$ by \cite[Prop. 4.4]{GM21}, \pref{thm:main} yields:

\begin{theorem}
   For any model considered in \cite{GHS20}, it holds that $d_w=d_f > 2$ and $d_s = 2$.
\end{theorem}

\begin{remark}\label{rem:models}
We remark that the lower bound $d_s \geq 2$ is established in \cite{GM21}, and the upper bound
$d_s \leq 2$ follows for any unimodular random planar graph where the degree
of the root has superpolynomial tails \cite{lee17a}
(which is true for each of these models; see \cite[\S 1.3]{GM21}).
The consequence $d_w = d_f$ is proved in \cite{GH20} for every model
except the uniform infinite Schynder-wood decorated triangulation.
This is for a technical reason underlying the identification
of $V(M_n)$ with a subset of $V(M)$ used in the proof of \cite[Lem. 1.11]{GHS20} (see \cite[Rem. 1.3]{GHS20} and
\cite[Rem. 2.11]{GH20}).
\end{remark}

\begin{proof}[Proof of \pref{cor:penpen}]
   Fix $\delta > 0$ and $R \geq 2$.
   Denote
   \begin{align*}
      \tilde{r} &\seteq (K \log n)^{-q} (R-2), \\
      \tilde{R} &\seteq (K \log n)^q R^{1+\delta}, \\
      n &\seteq \lceil \tilde{R}^K\rceil,
   \end{align*}
   and let $\cE_n$ be an event on which
   \pref{thm:rough}(1)--(4) and \eqref{eq:test-fun} hold.  Note that we can take $\Pr(\cE_n) \geq 1-O(R^{-\alpha'})$
   for some $\alpha' = \alpha'(\delta,K) > 0$.

   Assume now that $\cE_n$ holds.
   Then \eqref{eq:test-fun} and the Dirichlet principle \eqref{eq:dirichlet}
   give a test function $f : V(\cG) \to \R$ such that 
   \[
      f(B^{\cG}(0,R)) = 0, \qquad f(\partial_{\cG} B^{\cG}(0,R^{1+\delta}))=1, \qquad \energy^{\cG}(f) \leq 1/\kappa.
  \]
  \pref{thm:rough}(1) asserts that the restriction of $f$ to $B^{\cG}(0,R^{1+\delta})$
  gives a function $\tilde{f} : V(\cG_n) \to \R$ on which
   \[
      \tilde{f}(B^{\cG_n}(0,R))=0, \qquad \tilde{f}(\partial_{\cG} B^{\cG_n}(0,R^{1+\delta}))=1, \qquad
   \energy^{\cG_n}(\tilde{f}) \leq 1/\kappa.
   \]
   Without increasing the energy of $\tilde{f}$, we may assume that $\tilde{f}(V(\cG_n) \setminus B^{\cG_n}(0,R^{1+\delta}))=1$ as well.

   \smallskip

   By our choice of $\tilde{r}$ and $\tilde{R}$,
   \pref{thm:rough}(3) implies that
   \[
      \tilde{f} \circ \phi_n(B^{M_n}(\rho,\tilde{r})) = 0, \qquad \tilde{f} \circ \phi_n(\partial_{M_n} B^{M_n}(\rho,\tilde{R}))=1, \qquad \energy^{M_n}(\tilde{f} \circ \phi_n) \leq K' (\log R)/\kappa,
  \]
  where the last inequality is from \pref{thm:rough}(4), and $K' = K'(K,q,\delta)$.
  Now the Dirichlet principle \eqref{eq:dirichlet} yields
  \[
     \reff^{M_n}\left(\partial_{M_n} B^{M_n}(\rho_n,\tilde{r}) \leftrightarrow \partial_{M_n} B^{M_n}(\rho_n,\tilde{R})\right)
     \geq \frac{1}{K' (\log R)/\kappa},
  \]
  and from the graph isomorphism \pref{thm:rough}(2) and the fact that $n^{1/K} \geq \tilde{R}$, we conclude that
  \[
     \reff^{M}\left(\partial_{M} B^{M}(\rho,\tilde{r}) \leftrightarrow \partial_{M} B^{M}(\rho,\tilde{R})\right)
     \geq \frac{1}{K' (\log R)/\kappa}.
  \]
  Since this conclusion holds with probability at least $1-O(R^{-\alpha'})$, we conclude (using Borel-Cantelli
  as in the proof of \pref{cor:zeta-crt})
  that for every $\delta > 0$, almost surely eventually
  \[
     \reff^{M}\left(\partial_{M} B^{M}(\rho,R) \leftrightarrow \partial_{M} B^{M}(\rho,R^{1+\delta})\right)
     \geq R^{-\delta}.
  \]
  This yields $\tilde{\zeta}=0$, completing the proof.
\end{proof}

\subsection{Random walk driven by the GFF}
\label{sec:gff}

Denote by
$\bm{\eta} = \{ \eta_v : v \in \Z^2 \}$ the centered Gaussian process with $\eta_0 = 0$
and covariances
$\E[\eta_u \eta_v]=\green^{\Z^2}_{\Z^2 \setminus \{0\}}(u,v)$ for all $u,v \in \Z^2$,
where we recall the Green kernel from \pref{sec:resistance}.

Fix $\gamma > 0$, and 
define $G=\Z^2$ with $E(G) = \left\{ \{u,v\} \subseteq \Z^2 : \|u-v\|_1 = 1 \right\}$, and the conductances
\[
   \con^G(\{u,v\}) \seteq e^{\gamma(\eta_u-\eta_v)}, \qquad \{u,v\} \in E(G).
\]
Since the edge conductances only depend on the differences $\eta_u-\eta_v$, the
law of the conductances is translation invariant, and thus $(G,0,\con^G)$ is a reversible random network.
Moreover, $\E[1/\con^G_0] < \infty$ follows from the fact that $\eta_u-\eta_v$ is a Gaussian of variance 
of bounded variance for $\{u,v\} \in E(G)$.

\begin{theorem}\label{thm:gff}
   For the reversible random network $(G,0,\con^G)$, it holds that
   \begin{align*}
      d_f &= \begin{cases}
            2 + 2(\gamma/\gamma_c)^2 & \gamma \leq \gamma_c = \sqrt{\pi/2}, \\
            4 \gamma/\gamma_c & \textrm{otherwise,}
         \end{cases} \\
      \tilde{\zeta} &= 0, \\
      \tilde{\zeta}_0 &= 0, \\
      d_w &= d_f,  \\
      d_s &= 2.
   \end{align*}
\end{theorem}

Since the value of $d_f$ is elementary to calculate (see \cite{BDG20} for details),
and $\tilde{\zeta}_0 = 0$ is the content of Theorem 1.4(1.11) in \cite{BDG20},
we can apply \pref{thm:main}, and what remains is to verify $\tilde{\zeta}=0$.
To this end,
denote $\cS(N) \seteq [-N,N]^2 \cap \Z^2$.
Theorem 1.4 in \cite{BDG20} (specifically, equation (1.12))
establishes that
\begin{equation}\label{eq:bdg20}
   \liminf_{N \to \infty} \frac{\log \reff^G\left(0 \lra \Z^2 \setminus \cS(N)\right)}{(\log N/\log \log \log N)^{1/2}} > 0.
\end{equation}
In a moment, we will observe the following consequence of their argument.

\begin{lemma}\label{lem:bdg}
   There is some $c = c(\gamma) > 0$ such that 
   for every $N \geq 8$ sufficiently large, the following holds. For $1 \leq k \leq n-1$, where $n = \lfloor \log_8(N)\rfloor$,
   let $\cE_k$ denote the event
   \[
      \reff^G\left(\cS(2 \cdot 8^{n-k} N) \lra \Z^2 \setminus \cS(4 \cdot 8^{n-k} N) \right) \geq 
      e^{- (1/c) \log \log (N)}
   \]
   Then $\Pr(\cE_k) > c$ for each $k \in \{1,2\ldots,n-1\}$.
\end{lemma}

\begin{proof}[Proof of \pref{thm:gff}]
Fix some $\delta > 0$, and define $n' \seteq \lceil \delta \log (N)\rceil$.
Since the events $\{ \cE_k : 1 \leq k \leq n-1 \}$
involve disjoint sets of edges, they are independent, and we have
\[
   \Pr\left(\cE_{1} \vee \cE_{2} \vee \cdots \vee \cE_{n'}\right) \geq 1 - \left(1-c\right)^{n'} \geq 1 - N^{-\delta c}.
\]
Moreover, the series law for effective resistances gives
\[
   \reff^G\left(\cS(2 \cdot 8^{n-1} N) \lra \Z^2 \setminus \cS(4 \cdot 8^{n-n'} N)\right)
   \geq \1_{\{\cE_1 \vee \cE_2 \vee \cdots \vee \cE_{n'}\}} 
   e^{-(1/c) \log \log (N)}.
\]
Thus an application of Borel-Cantelli yields:
Almost surely eventually,
\[
   \reff^G\left(\cS(R) \lra \Z^2 \setminus \cS(R^{1+\delta})\right) \geq R^{-\delta}.
\]
Since $\delta > 0$ could be chosen arbitrarily small, we conclude that $\tilde{\zeta}=0$.
\end{proof}

Let us finally indicate how \pref{lem:bdg} follows from the arguments in \cite{BDG20}.
The authors define in \cite[Eq. (5.79)]{BDG20} an event $F_k$ such that $\Pr(F_k) > c_0 > 0$, and
on $F_k$ it holds that
\[
      \reff^G\left(\cS(2 \cdot 8^{n-k} N) \lra \Z^2 \setminus \cS(4 \cdot 8^{n-k} N) \right) \geq 
      e^{-2\gamma(\Delta_k-S_k)} e^{-3 \hat{\mathfrak{c}} \log \log (N)},
\]
where $S_k$, $\Delta_k$, and $\1_{F_k}$ are mutually independent random variables, $c_0 > 0$
is a universal constant, and $\hat{\mathfrak{c}} = \hat{\mathfrak{c}}(\gamma) > 0$.
We remark that in \cite{BDG20}, the exponent is given as $-3 \hat{\mathfrak{c}} \log (b^k)$ (where $b=8$), but
the correct value (as stated in \cite[Lem. 4.13]{BDG20}) is $-3 \hat{\mathfrak{c}} \log \log (b^k)$.
(And, indeed, the correct quantitative dependence is needed to conclude \eqref{eq:bdg20}.)

Moreover, it holds that the law of $S_k$ is symmetric, and (5.74) in \cite{BDG20}
asserts that for some constants $c_1,c_2 > 0$
and all $t \geq 0$,
\[
   \Pr(\Delta_k \geq c_1+t)\leq e^{-c_2 t^2}.
\]
We conclude that for some number $C > 0$ chosen sufficiently large,
\[
   \Pr\left(\reff^G\left(\cS(2 \cdot 8^{n-k} N) \lra \Z^2 \setminus \cS(4 \cdot 8^{n-k} N) \right) \geq 
   e^{-2C\gamma} e^{-3 \hat{\mathfrak{c}} \log \log (N)}\right) \geq \frac{c_0}{4} > 0,
\]
thereby verifying \pref{lem:bdg}.

\subsection*{Acknowledgements}

The author thanks Jian Ding, Ewain Gwynne, Mathav Murugan, and Asaf Nachmias for useful feedback and pointers to the literature,
and Jian for suggesting that our main theorem is applicable to the model in \cite{BDG20}.
This work was partially funded by a Simons Investigator Award.

\bibliographystyle{alpha}
\bibliography{diffusive}

\end{document}